\documentclass[11pt,a4paper]{article}
\usepackage{caption}
\usepackage{epsf,epsfig,amsfonts,amsgen,amsmath,amstext,amsbsy,amsopn,amsthm}
\usepackage{color}
\usepackage{mathrsfs}
\usepackage{tikz}
\usepackage{enumerate}
\usepackage{url}
\usepackage{enumitem}
\usepackage{lineno}
\usepackage{latexsym,amsmath,amssymb,amsfonts,epsfig,graphicx,cite,psfrag}
\usepackage{eepic,color,colordvi,amscd}

\newtheorem{theorem}{Theorem}[section]
\newtheorem{corollary}[theorem]{Corollary}

\newtheorem{conjecture}[theorem]{Conjecture}

\newtheorem{lemma}[theorem]{Lemma}
\newtheorem{proposition}[theorem]{Proposition}

\newtheorem{claim}{Claim}

\newcommand{\1}{{\uppercase\expandafter{\romannumeral1}}}
\newcommand{\2}{{\uppercase\expandafter{\romannumeral2}}}
\newcommand{\3}{{\uppercase\expandafter{\romannumeral3}}}
\newcommand{\4}{{\uppercase\expandafter{\romannumeral4}}}

\begin{document}
\title{A clique version of the Erd\H{o}s-Gallai stability theorems}

\author{
Jie Ma\thanks{School of Mathematical Sciences, University of Science and Technology of China, Hefei 230026, China.
Email: jiema@ustc.edu.cn. Supported in part by National Natural Science Foundation of China grant 11622110,
the project ``Analysis and Geometry on Bundles'' of Ministry of Science and Technology of the People's Republic of China,
and Anhui Initiative in Quantum Information Technologies grant AHY150200.}
\and
Long-Tu Yuan\thanks{School of Mathematical Sciences, East China Normal University, 500 Dongchuan Road, Shanghai 200240, China.
Email: ltyuan@math.ecnu.edu.cn. Supported in part by National Natural Science Foundation of China grant 11901554.}
}

\date{}
\maketitle
\begin{abstract}
Combining P\'{o}sa's rotation lemma with a technique of Kopylov in a novel approach,
we prove a generalization of the Erd\H{o}s-Gallai theorems on cycles and paths.
This implies a clique version of the Erd\H{o}s-Gallai stability theorems and also provides alternative proofs for some recent results.
\end{abstract}



\section{Introduction}
\noindent
The {\it circumference} $c(G)$ of a graph $G$ is the length of a longest cycle in $G$.
For $s\geq 2$, let $N_s(G)$ denote the number of unlabeled copies of the clique $K_s$ in $G$.
For integers $n\geq k\geq 2a$, let $H(n,k,a)$ be the $n$-vertex graph whose vertex set is partitioned into three sets $A, B, C $ such that $|A|=a, |B|=n-k+a$ and $|C|=k-2a$
and the edge set consists of all edges between $A$ and $B$ together with all edges in $A\cup C$.
Let $h_s(n,k,a)=N_s(H(n,k,a))$.

The celebrated Erd\H{o}s-Gallai theorem \cite{erdHos1959maximal} states that any $n$-vertex graph $G$ with $c(G)<k$ has at most $\frac{k-1}{2}(n-1)$ edges.
This was improved by Kopylov \cite{Kopylov1977} by showing that any $n$-vertex $2$-connected graph $G$ with $c(G)<k$ has at most $\max\{h_2(n,k,2)), h_2(n,k,\lfloor\frac{k-1}{2}\rfloor)\}$ edges.
Combined with the results in \cite{Furedi2016}, F\"{u}redi, Kostochka, Luo and Verstra\"{e}te \cite{Furedi2018} proved a stability version of Kopylov's theorem,
which says that for any 2-connected graph $G$ with $c(G)<k$,
if $e(G)$ is close to the above maximum number from Kopylov's theorem, then $G$ must be a subgraph of some well-specified graphs.
This was further extended in \cite{MN2019} to stability results of $2$-connected graphs of any given minimum degree.
On the other hand, Luo \cite{Luo2018} generalized Kopylov's theorem by showing that
the number of $s$-cliques in any $n$-vertex 2-connected graph $G$ with $c(G)<k$
is at most $\max\{h_s(n,k,2),h_s(n,k,\lfloor\frac{k-1}{2}\rfloor)\}$.

The aim of this paper is to study a new approach and provide some potential tools in this line of research.
Following this approach, our main result, Theorem~\ref{Theorem 2-connected}, considers a general stability setting.
To get into the statement, it requires an entangled family of some specified graphs which we will define in Section 2.
However, we would like to point out that using Theorem \ref{Theorem 2-connected}, one can not only derive alternate proofs of many recent results in \cite{Furedi2016,Furedi2018,MN2019},
but also infer some new results.
One such new result is the following stability result of Luo's theorem \cite{Luo2018} on the number of $s$-cliques,
which also can be viewed as a clique version of one of the main results in \cite{Furedi2016} (see Theorem 1.4 therein).

\begin{theorem}\label{corollary 2-connected large n}
Let $k\geq 5$, $\ell=\lfloor (k-1)/2\rfloor$, $2\leq s\leq \max\{2, \ell-1\}$ and $n\geq n_0(\ell)$ be integers.\footnote{We remark that we do not pursue a precise expression on the constant $n_0(\ell)$, which needs to be quite large in the proof. Instead, we pose a conjecture in the end of this paper which would help extend Theorem~\ref{corollary 2-connected large n} to all values of $n$.}
Let $G$ be an $n$-vertex $2$-connected graph with $c(G)<k$.
Then $N_s(G)\leq h_s(n,k,\ell-1)$ unless
\begin{itemize}
\item [(a)] $s=3$ and $k\in \{9,10\}$,
\item [(b)] $k=2\ell+1$, $k\neq 7$, and $G\subseteq H(n,k,\ell)$, or
\item [(c)] $k=2\ell+2$ or $k=7$, and $G-A$ is a star forest for some $A\subseteq V(G)$ of size at most $\ell$.\footnote{A {\it star forest} is a graph in which every component is a star.}
\end{itemize}
\end{theorem}

We would like to mention that Theorem~\ref{Theorem 2-connected} can in fact imply a refined version of Theorem~\ref{corollary 2-connected large n}.




The organization of this paper is as follows.
In Section~\ref{notation}, we give a formal definition of a family of graphs for the use of our characterization.
In Section~\ref{sec posa lemma}, we prove the key lemma for our proofs which builds on an integration of P\'{o}sa's rotation lemma and Kopylov's proof in \cite{Kopylov1977}.
In Section~\ref{main result}, we prove our main result Theorem~\ref{Theorem 2-connected}.
In Section~\ref{corollary}, we show how to use Theorem~\ref{Theorem 2-connected} to deduce Theorem~\ref{corollary 2-connected large n} as well as some main results in \cite{Furedi2016,Furedi2018,MN2019}.
We also pose a related conjecture for further research to conclude this paper.
For the proof ideas, we would like to suggest readers to first look through the odd $k$ case, in which the family defined in Section~\ref{notation} as well as the proofs will be much more concise.
Throughout the rest of the paper, let $k\geq 5$ be an integer and $\ell=\lfloor(k-1)/2\rfloor$.

\section{Notation and a family of graphs}\label{notation}

\subsection{Notation}
The general notation used in this paper is standard (see, e.g., \cite{bollobas1978}).
For a graph $G$, let $\omega(G)$ be the order of a maximum clique in $G$.
For disjoint subsets $A, B\subseteq V(G)$, we denote $G(A,B)$ to be the induced bipartite subgraph of $G$ with parts $A,B$. 
Let $E(A,B)=E(G(A,B))$ for short.
When defining a graph, we will only specify these adjacent pairs of vertices, that says,
if a pair $\{a,b\}$ is not discussed as a possible edge, then it is assumed to be a non-edge.

Denote by $N_G(x)$ the set of neighbors of $x$ in $G$ and let $d_G(x)$ be the size of $N_G(x)$.
For $U\subseteq V(G)$, let $N_U(x)=N_G(x)\cap U$ and $d_U(x)=|N_U(x)|$.
Let $P=x_1x_2\cdots x_m$ be a path in $G$.
For $x\in V(G)$, let $N_P(x)=N_G(x)\cap V(P)$ and $N_P[x]=N_P(x)\cup \{x\}$, with $d_P(x):=|N_P(x)|$.
For $x_i,x_j\in V(P)$, we use $x_iPx_j$ to denote the subpath of $P$ between $x_i$ and $x_j$.
For $x\in V(P)$, denote $x^-$ and $x^+$ to be the immediate predecessor and immediate successor of $x$ on $P$, respectively.
For $S \subseteq V (P)$, let $S^+=\{x^+:x\in S\}$ and $S^-=\{x^-:x\in S\}$.
We call $(x_i,x_j)_P$ a {\it crossing pair} of $P$ if $i<j$, $x_i\in N_P(x_m)$ and $x_j\in N_P(x_1)$.
If there is no ambiguity, we write this pair as $(i,j)$ for short.
We call a path a {\it crossing path} if it has a crossing pair.
Let $j-i-1$ be the {\it length} of the crossing pair $(i,j)$.
A crossing pair $(i,j)$ is {\it minimal} in $P$ if $x_h\notin N_P(x_1)\cup N_P(x_m)$ for each $i<h<j$.
For $S\subseteq V(G)$, we call $P$ an $S${\it-path} if $x_1,x_m\in S$.

For an integer $\alpha$ and a graph $G$, the {\it $\alpha$-disintegration} of $G$,
denoted by $H(G,\alpha)$,
is the graph obtained from $G$ by recursively deleting vertices of degree at most $\alpha$ until that the resulting graph has no such vertex.\footnote{One can see that $H(G,\alpha)$ is unique in $G$ and has minimum degree at least $\alpha+1$ (if non-empty).}

\subsection{A family of graphs}

Let $m\geq k\geq 5$ and $1\leq r\leq \ell$ be integers.
We now devote the rest of this subsection to the definition of a family $\mathcal{F}(m,k,r)$ of some delicate graphs $F$,
each of which has a Hamilton path
and satisfies $|V(F)|=m$ and $c(F)<k$.\footnote{For the parameter $r$, roughly speaking we may view it as something close to $\omega(F)$, though its own meaning will be clear in the proof of Lemma~\ref{extend posa lemma}. Readers may treat the coming lengthy definition as a handout and skip to next sections.}
We divide $\mathcal{F}(m,k,r)$ into the following four classes, namely Types \1, \2, \3 and \4.
Along the way, we also define some very special graphs (see Figure 1).

\medskip

\noindent{\bf Type \1:} Let $k=2\ell+1$ be odd and $r\leq \ell-1$. Each graph $F\in \mathcal{F}(m,k,r)$ of Type \1 satisfies:

\begin{itemize}
  \item $V(F)=A\cup B\cup C\cup D$, 
  \item $F[A]$ and $F[B]$ are cliques on $r$ vertices,
  \item $F[C]$ is empty with $|C|=\ell-r+1$,
  \item $F[D]$ is empty when $|C|\geq 3$, and $F[D]$ is  a path when $|C|=2$,\footnote{An isolated vertex will also be viewed as a (trivial) path in this paper.}
  \item each vertex in $A\cup B$ is adjacent to each vertex in $C$, and
  \item $F[C\cup D]$ is a $C$-path.
\end{itemize}

\noindent{\bf Type \2:} Let $k=2\ell+2$ be even and $r\leq \ell-1$. Each graph $F\in \mathcal{F}(m,k,r)$ of Type \2 satisfies:
\begin{itemize}
  \item $V(F)=A\cup B\cup C\cup D$, 
  \item $|A|\in\{r,r+1\}$ and $F[B]$ is a clique on $r$ vertices,
  \item $F[C]$ is empty with $|C|=\ell-r+1$,
  \item $F[D]$ is a path when $|C|=2$, and $F[D]$ consists of at most two independent edges and some isolated vertices when $|C|\geq 3$  such that one of the following holds:
  \begin{itemize}
  \item $F[D]$ is empty when $|A|=r+1$,
  \item $F[D]$ contains a unique edge when $|A|=r$, or
  \item $F[D]$ consists of two independent edges when $|A|=r=\ell-2$.
\end{itemize}
 \item each vertex in $A$ has degree exactly $\ell$ in $F[A\cup C]$\footnote{Note that if $r=1$ and $|A|=2$, then $F[A]=K_2$ (by the fact that $F$ contains a Hamilton path).} and each vertex in $B$ has degree exactly $\ell$ in $F[B\cup C]$, and
  \item $F[C\cup D]$ is a $C$-path satisfying that if $|A|=r+1$ then each end-vertex of $F[C\cup D]$ is adjacent to some vertex of $A$.
   In particular, we denote the graph with $|A|=r=\ell-2$ and $|D|=3$ by $F_0(m,k,r)$, the graph with $|A|=r=\ell-2$ and $|D|=4$ by $F_4(m,k,r)$, and the graph with $F[A]$ being a star on three vertices by $F_5(m,k,2)$.
  \end{itemize}

\noindent{\bf Type \3:} Let $k=2\ell+2$ be even and $r\leq \ell-1$. Each graph $F\in \mathcal{F}(m,k,r)$ of Type \3 satisfies:
\begin{itemize}
  \item $V(F)=A\cup B\cup C\cup D$,
  \item $F[A]$ and $F[B]$ are cliques on $r$ vertices,
  \item $F[C]$ is empty with $|C|=\ell-r+1$,
  \item $F[D]$ is empty when $|C|\geq 3$, and $F[D]$ consists of a path and an isolated vertex when $|C|=2$,
  \item each vertex in $A\cup B$ is adjacent to each vertex in $C$, and
  \item $F[C\cup D]$ consists of at most two vertex-disjoint paths such that one of the following holds:
 \begin{itemize}
   \item $F[C\cup D]$ consists of a $C$-path and an isolated vertex $x\in D$ such that
    $x$ is adjacent to exactly two vertices $x_1,x_2$ of $A$ (denote this graph by $F_1(m,k,r)$),
   \item $F[C\cup D]$  is a path with the end-vertex $y\in D$ such that $y$ is an isolated vertex in $F[D]$ and is adjacent to exactly one vertex $y_1$ in $A$ (denote this graph by $F_2(m,k,r)$), or
\item $F[C\cup D]$ consists of a path with distinct end-vertices $z,z^\prime \in D$ and a path with end-vertices in $C$ satisfying that $|D|=\ell-r+1$ and $z, z'$ is adjacent to exactly one vertex $z_1, z_1^\prime$ in $A$, respectively, with $z_1\neq z_1^\prime$ (denote this family of graphs by $\mathcal{F}_3(m,k,r)$). 
  \end{itemize}
\end{itemize}

\noindent{\bf Type \4:} Let $k=2\ell+2$ be even and $r=\ell$. Each graph $F\in \mathcal{F}(m,k,r)$ of Type \4 satisfies:

\begin{itemize}
\item $V(F)=A\cup B\cup C$, 
\item $F[A]$ and $F[B]$ are cliques on $\ell-1$ vertices, and
\item $F[C]$ induces a cycle with three distinct vertices $w_1, w_2, w$ such that $w_1w_2\in E(F[C])$, $ww_i\notin E(F[C])$ for $i\in \{1,2\}$, $w_1$ is adjacent to each vertex of $A$, $w_2$ is adjacent to each vertex of $B$, and $w$ is adjacent to each vertex of $A\cup B$.
\end{itemize}

\begin{center}
\begin{tikzpicture}[scale = 0.5]
\filldraw[fill=black] (0,0) circle(2pt);
\filldraw[fill=black] (0.5,-2) circle(2pt);
\filldraw[fill=black] (-0.5,-2) circle(2pt);
\filldraw[fill=black] (-1,0) circle(2pt);
\filldraw[fill=black] (1.5,-2) circle(2pt);
\filldraw[fill=black] (2,0) circle(2pt);
\filldraw[fill=black] (-2,-2) circle(2pt);
\filldraw[fill=black] (-3,-2) circle(2pt);
\filldraw[fill=black] (-4,-2) circle(2pt);
\filldraw[fill=black] (3,-2) circle(2pt);
\filldraw[fill=black] (4,-2) circle(2pt);
\filldraw[fill=black] (5,-2) circle(2pt);

\draw (0,0) -- (0.5,-2);
\draw (-1,0) -- (-0.5,-2);
\draw  (0,0) -- (-0.5,-2);
\draw  (0,0) -- (0.5,-2);
\draw  (1.5,-2) -- (0.5,-2);
\draw  (1.5,-2) -- (2,0);

\draw  (-2,-2) -- (2,0);
\draw  (-2,-2) -- (-1,0);
\draw  (-2,-2) -- (0,0);
\draw  (-3,-2) -- (2,0);
\draw  (-3,-2) -- (-1,0);
\draw  (-3,-2) -- (0,0);
\draw  (-4,-2) -- (2,0);
\draw   (-4,-2) -- (-1,0);
\draw  (-4,-2) -- (0,0);

\draw  (3,-2) -- (2,0);
\draw  (3,-2) -- (-1,0);
\draw  (3,-2) -- (0,0);
\draw  (4,-2) -- (2,0);
\draw  (4,-2) -- (-1,0);
\draw  (4,-2) -- (0,0);
\draw  (5,-2) -- (2,0);
\draw  (5,-2) -- (-1,0);
\draw  (5,-2) -- (0,0);

\draw (0.5,0) ellipse (2 and 0.5);
\draw (-3,-2) ellipse (1.5 and 0.5);
\draw (4,-2) ellipse (1.5 and 0.5);
\draw (0.5,-2) ellipse (1.5 and 0.5);

\draw  (3,-2) -- (4,-2);
\draw  (4,-2) -- (5,-2);
\draw  (3,-2)..  controls (4,-2.5).. (5,-2);
\draw  (-2,-2) -- (-3,-2);
\draw  (-4,-2) -- (-3,-2);

\draw  (-4,-2)..  controls (-3,-2.5).. (-2,-2);

\draw node at (0.5,-3){$D$};
\draw node at (-0.5,-3){$v$};
\draw node at (-3,-3){$A$};
\draw node at (3.5,-3){$B$};
\draw node at (1,1){$C$};
\draw node at (0,1){$v_2$};
\draw node at (-1,1){$v_1$};

\draw node at (0.5,-4){$F_0(12,12,3)$};

\filldraw[fill=black] (10,0) circle(2pt);
\filldraw[fill=black] (10.5,-2) circle(2pt);
\filldraw[fill=black] (9.5,-2) circle(2pt);
\filldraw[fill=black] (11.5,-2) circle(2pt);
\filldraw[fill=black] (11,0) circle(2pt);
\filldraw[fill=black] (12,0) circle(2pt);
\filldraw[fill=black] (8.5,-2) circle(2pt);
\filldraw[fill=black] (7.5,-2) circle(2pt);
\filldraw[fill=black] (6.5,-2) circle(2pt);
\filldraw[fill=black] (13,-2) circle(2pt);
\filldraw[fill=black] (14,-2) circle(2pt);
\filldraw[fill=black] (15,-2) circle(2pt);

\draw  (10,0) -- (10.5,-2);
\draw  (11.5,-2) -- (11,0);
\draw (11,0) -- (10.5,-2);
\draw  (11.5,-2) -- (12,0);

\draw (13,-2)-- (12,0);
\draw (14,-2)-- (12,0);
\draw (15,-2)-- (12,0);
\draw (13,-2)-- (11,0);
\draw (14,-2)-- (11,0);
\draw (15,-2)-- (11,0);
\draw (13,-2)-- (10,0);
\draw (14,-2)-- (10,0);
\draw (15,-2)-- (10,0);

\draw (6.5,-2)-- (12,0);
\draw (7.5,-2)-- (12,0);
\draw (8.5,-2)-- (12,0);
\draw (6.5,-2)-- (11,0);
\draw (7.5,-2)-- (11,0);
\draw (8.5,-2)-- (11,0);
\draw (6.5,-2)-- (10,0);
\draw (7.5,-2)-- (10,0);
\draw (8.5,-2)-- (10,0);

\draw (9.5,-2)-- (8.5,-2);

\draw (11,0) ellipse (1.5 and 0.5);
\draw (7.5,-2) ellipse (1.5 and 0.5);
\draw (14,-2) ellipse (1.5 and 0.5);
\draw (10.5,-2) ellipse (1.5 and 0.5);

\draw  (13,-2) -- (14,-2);
\draw  (14,-2) -- (15,-2);
\draw  (13,-2)..  controls (14,-2.5).. (15,-2);
\draw  (6.5,-2) -- (7.5,-2);
\draw  (8.5,-2) -- (7.5,-2);
\draw  (8.5,-2)..  controls (7.5,-2.5).. (6.5,-2);

\draw  (9.5,-2)..  controls (8,-3).. (6.5,-2);

\draw node at (10.5,-3){$D$};
\draw node at (9.5,-3){$x$};
\draw node at (8.5,-3){$x_1$};
\draw node at (6.5,-3){$x_2$};
\draw node at (7.5,-3){$A$};
\draw node at (14,-3){$B$};
\draw node at (11,1){$C$};

\draw node at (11,-4){$F_1(12,12,3)$};

\filldraw[fill=black] (20,0) circle(2pt);
\filldraw[fill=black] (20.5,-2) circle(2pt);
\filldraw[fill=black] (19.5,-2) circle(2pt);
\filldraw[fill=black] (21.5,-2) circle(2pt);
\filldraw[fill=black] (21,0) circle(2pt);
\filldraw[fill=black] (22,0) circle(2pt);
\filldraw[fill=black] (18.5,-2) circle(2pt);
\filldraw[fill=black] (17.5,-2) circle(2pt);
\filldraw[fill=black] (16.5,-2) circle(2pt);
\filldraw[fill=black] (23,-2) circle(2pt);
\filldraw[fill=black] (24,-2) circle(2pt);
\filldraw[fill=black] (25,-2) circle(2pt);

\draw (20,0) -- (20.5,-2);

\draw  (21.5,-2) -- (21,0);
\draw (21,0) -- (20.5,-2);
\draw  (21.5,-2) -- (22,0);

\draw (23,-2)-- (22,0);
\draw (24,-2)-- (22,0);
\draw (25,-2)-- (22,0);
\draw (23,-2)-- (21,0);
\draw (24,-2)-- (21,0);
\draw (25,-2)-- (21,0);
\draw (23,-2)-- (20,0);
\draw (24,-2)-- (20,0);
\draw (25,-2)-- (20,0);

\draw (16.5,-2)-- (22,0);
\draw (17.5,-2)-- (22,0);
\draw (18.5,-2)-- (22,0);
\draw (16.5,-2)-- (21,0);
\draw (17.5,-2)-- (21,0);
\draw (18.5,-2)-- (21,0);
\draw (16.5,-2)-- (20,0);
\draw (17.5,-2)-- (20,0);
\draw (18.5,-2)-- (20,0);

\draw (19.5,-2)-- (18.5,-2);

\draw (19.5,-2)-- (20,0);

\draw (21,0) ellipse (1.5 and 0.5);
\draw (17.5,-2) ellipse (1.5 and 0.5);
\draw (24,-2) ellipse (1.5 and 0.5);
\draw (20.5,-2) ellipse (1.5 and 0.5);

\draw  (23,-2) -- (24,-2);
\draw  (24,-2) -- (25,-2);
\draw  (23,-2)..  controls (24,-2.5).. (25,-2);
\draw  (16.5,-2) -- (17.5,-2);
\draw  (18.5,-2) -- (17.5,-2);
\draw  (18.5,-2)..  controls (17.5,-2.5).. (16.5,-2);

\draw node at (20.5,-3){$D$};

\draw node at (19.5,-3){$y$};
\draw node at (18.5,-3){$y_1$};
\draw node at (17.5,-3){$A$};
\draw node at (24,-3){$B$};
\draw node at (21,1){$C$};
\draw node at (20,1){$y_2$};

\draw node at (21,-4){$F_2(12,12,3)$};

\end{tikzpicture}
\begin{tikzpicture}[scale = 0.5]
\filldraw[fill=black] (0,0) circle(2pt);
\filldraw[fill=black] (0.5,-2) circle(2pt);
\filldraw[fill=black] (-0.5,-2) circle(2pt);
\filldraw[fill=black] (1,0) circle(2pt);
\filldraw[fill=black] (1.5,-2) circle(2pt);
\filldraw[fill=black] (2,0) circle(2pt);
\filldraw[fill=black] (-2,-2) circle(2pt);
\filldraw[fill=black] (-3,-2) circle(2pt);
\filldraw[fill=black] (-4,-2) circle(2pt);
\filldraw[fill=black] (3,-2) circle(2pt);
\filldraw[fill=black] (4,-2) circle(2pt);
\filldraw[fill=black] (5,-2) circle(2pt);

\draw (0,0) -- (0.5,-2);
\draw  (0,0) -- (-0.5,-2);
\draw  (0,0) -- (0.5,-2);
\draw  (1.5,-2) -- (1,0);
\draw  (1.5,-2) -- (2,0);

\draw  (-2,-2) -- (2,0);
\draw  (-2,-2) -- (1,0);
\draw  (-2,-2) -- (0,0);
\draw  (-3,-2) -- (2,0);
\draw  (-3,-2) -- (1,0);
\draw  (-3,-2) -- (0,0);
\draw  (-4,-2) -- (2,0);
\draw   (-4,-2) -- (1,0);
\draw  (-4,-2) -- (0,0);

\draw  (3,-2) -- (2,0);
\draw  (3,-2) -- (1,0);
\draw  (3,-2) -- (0,0);
\draw  (4,-2) -- (2,0);
\draw  (4,-2) -- (1,0);
\draw  (4,-2) -- (0,0);
\draw  (5,-2) -- (2,0);
\draw  (5,-2) -- (1,0);
\draw  (5,-2) -- (0,0);

\draw (1,0) ellipse (1.5 and 0.5);
\draw (-3,-2) ellipse (1.5 and 0.5);
\draw (4,-2) ellipse (1.5 and 0.5);
\draw (0.5,-2) ellipse (1.5 and 0.5);

\draw  (3,-2) -- (4,-2);
\draw  (4,-2) -- (5,-2);
\draw  (3,-2)..  controls (4,-2.5).. (5,-2);
\draw  (-2,-2) -- (-3,-2);
\draw  (-4,-2) -- (-3,-2);

\draw  (-4,-2)..  controls (-3,-2.5).. (-2,-2);
\draw  (-2,-2) -- (-0.5,-2);

\draw  (-4,-2)..  controls (-1.75,-3).. (0.5,-2);

\draw node at (0.5,-3){$z$};
\draw node at (1.5,-3){$D$};
\draw node at (-0.5,-3){$z^\prime$};
\draw node at (-2,-3){$z_1^\prime$};
\draw node at (-3,-3){$A$};
\draw node at (-4,-3){$z_1$};
\draw node at (4,-3){$B$};
\draw node at (0,1){$z_2(z_2^\prime)$};
\draw node at (1.5,1){$C$};

\draw node at (0.5,-4){$F\in\mathcal{F}_3(12,12,3)$};

\filldraw[fill=black] (13,0) circle(2pt);
\filldraw[fill=black] (13.5,-2) circle(2pt);
\filldraw[fill=black] (12.5,-2) circle(2pt);
\filldraw[fill=black] (11.5,-2) circle(2pt);
\filldraw[fill=black] (14.5,-2) circle(2pt);
\filldraw[fill=black] (11,0) circle(2pt);
\filldraw[fill=black] (15,0) circle(2pt);
\filldraw[fill=black] (10,-2) circle(2pt);
\filldraw[fill=black] (9,-2) circle(2pt);
\filldraw[fill=black] (8,-2) circle(2pt);
\filldraw[fill=black] (16,-2) circle(2pt);
\filldraw[fill=black] (17,-2) circle(2pt);
\filldraw[fill=black] (18,-2) circle(2pt);

\draw (11.5,-2) -- (12.5,-2);
\draw (11,0) -- (11.5,-2);
\draw  (13,0) -- (12.5,-2);
\draw  (13,0) -- (13.5,-2);
\draw  (13.5,-2) -- (14.5,-2);
\draw  (14.5,-2) -- (15,0);

\draw  (16,-2) -- (11,0);
\draw  (16,-2) -- (13,0);
\draw  (16,-2) -- (15,0);
\draw  (17,-2) -- (11,0);
\draw  (17,-2) -- (13,0);
\draw  (17,-2) -- (15,0);
\draw  (18,-2) -- (11,0);
\draw   (18,-2) -- (13,0);
\draw  (18,-2) -- (15,0);

\draw  (8,-2) -- (11,0);
\draw  (8,-2) -- (13,0);
\draw  (8,-2) -- (15,0);
\draw  (9,-2) -- (11,0);
\draw  (9,-2) -- (13,0);
\draw  (9,-2) -- (15,0);
\draw  (10,-2) -- (11,0);
\draw   (10,-2) -- (13,0);
\draw  (10,-2) -- (15,0);

\draw (13,0) ellipse (2.5 and 0.5);
\draw (9,-2) ellipse (1.5 and 0.5);
\draw (17,-2) ellipse (1.5 and 0.5);
\draw (13,-2) ellipse (2 and 0.5);

\draw  (16,-2) -- (17,-2);
\draw  (17,-2) -- (18,-2);
\draw  (16,-2)..  controls (17,-2.5).. (18,-2);
\draw  (8,-2) -- (9,-2);
\draw  (10,-2) -- (9,-2);

\draw  (10,-2)..  controls (9,-2.5).. (8,-2);

\draw node at (13,-3){$D$};
\draw node at (9,-3){$A$};
\draw node at (17,-3){$B$};
\draw node at (13,1){$C$};
\draw node at (13,-4){$F_4(13,12,3)$};

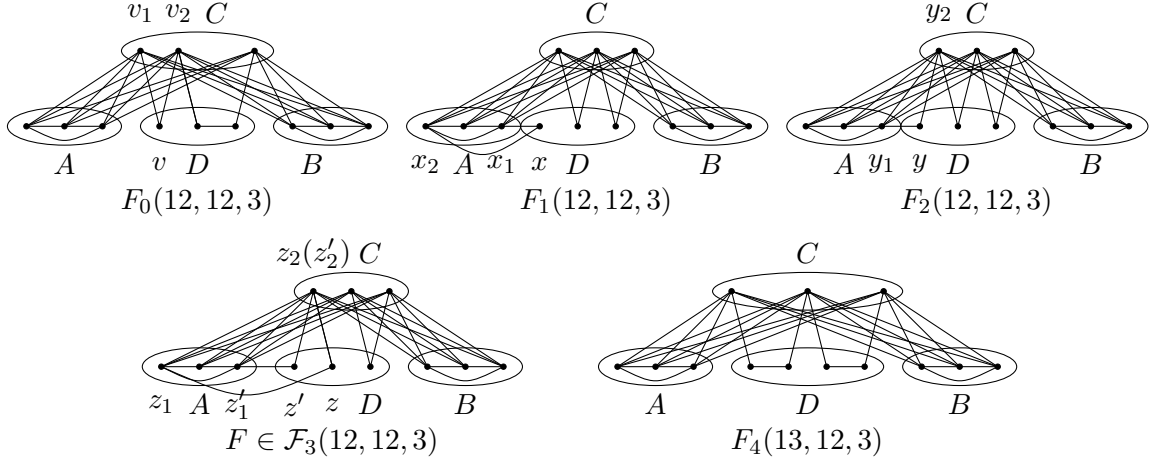
\draw node at (7.5,-6){Figure 1. Some special graphs in the family $\mathcal{F}(m,k,r)$};
\end{tikzpicture}
\end{center}

\noindent We point out that by definition, there is a Hamilton path in each $F\in \mathcal{F}(m,k,r)$ starting from $A$ and ending at $B$.
Also, if $k$ is odd, then all graphs in $\mathcal{F}(m,k,r)$ have Type \1.
Furthermore, $F_4(m,k,\ell-2)=F_4(k+1,k,\ell-2)$ is the only graph in $\mathcal{F}(m,k,r)$ with $m>k$ and $r\leq \ell-2$.

\section{A generalization of P\'{o}sa's lemma}\label{sec posa lemma}

The following well-known lemma is due to P\'{o}sa \cite{Posa1962} and is extensively used in extremal graph theory.

\begin{lemma}[P\'{o}sa \cite{Posa1962}]\label{posa lemma}
Let $G$ be a $2$-connected graph and $P=x_1x_2\cdots x_m$ be a path in $G$.
Then $G$ contains a cycle of length at least $\min\{m,d_P(x_1)+d_P(x_m)\}$.
Moreover, let $i$ be the minimum integer such that $x_i$ is adjacent to $x_m$ and $j$ be the maximum integer such that $x_j$ is adjacent to $x_1$.
If $j=i$, then $G$ contains a cycle of length at least $\min\{m,d_P(x_1)+d_P(x_m)+1\}$.
In addition, if $j<i$, then $G$ contains a cycle of length at least $\min\{m,d_P(x_1)+d_P(x_m)+2\}$.
\end{lemma}

The following lemma, which combines the ideas of P\'{o}sa's lemma \cite{Posa1962} and Kopylov's work \cite{Kopylov1977}, is the key technical part in the proof of our main theorem.
As we shall see in later sections, under contain circumstances, to find some particular useful subgraphs will be enough to determine the whole structure of graphs.
Recall that $k\geq5$ and $\ell=\lfloor(k-1)/2\rfloor$.

\begin{lemma}\label{extend posa lemma}
Let $G$ be a $2$-connected graph with $c(G)<k$.
Suppose that the $(\ell-1)$-disintegration $H$ of $G$ is non-empty and the longest $H$-path in $G$ has $m\geq k$ vertices.
Then $G$ contains a subgraph $F\in \mathcal{F}(m,k,r)$ for some $r\leq \ell$.
\end{lemma}

\noindent{\it Proof.}
We devote the rest of this section to the proof of this lemma.
Suppose to the contrary that $G$ dose not contain any subgraph in $\mathcal{F}(m,k,r)$ with $r\leq \ell$.
Let $\mathcal{P}$ be the family of all longest $H$-paths in $G$.
We process by showing a sequence of claims in what follows.

\begin{claim}\label{3basiclemma}
Every $P=x_1x_2\cdots x_m\in\mathcal{P}$ satisfies the following properties.
\begin{itemize}
    \item [$(\romannumeral1)$] $N_H(x_1)\subseteq N_P(x_1)$ and $N_H(x_m)\subseteq N_P(x_m)$,
    \item [$(\romannumeral2)$] $d_P(x_1)\geq d_H(x_1)\geq \ell$ and $d_P(x_m)\geq d_H(x_m)\geq\ell$, and
    \item [$(\romannumeral3)$] $N_P^-(x_1)\cap N_P[x_m]=\emptyset$ and $N_P^+(x_{m})\cap N_P[x_1]=\emptyset$.
\end{itemize}
\end{claim}

\begin{proof}
Suppose to the contrary that there exists $y\in N_H(x_1)\setminus N_P(x_1)$.
Clearly, $yx_1Px_m$ is an $H$-path longer than $P$, a contradiction.
Therefore, we have $N_H(x_1)\subseteq N_P(x_1)$.
Similarly, we have $N_H(x_m)\subseteq N_P(x_m)$.
Note that $H$ is the $(\ell-1)$-disintegration of $G$ and $H$ is non-empty.
It follows that $d_P(x_1)\geq d_H(x_1)\geq \ell$ and $d_P(x_m)\geq d_H(x_m)\geq \ell$.

Suppose to the contrary that $N_P^-(x_1)\cap N_P[x_{m}]\neq\emptyset$.
Let $x_i$ be a vertex in $N_P^-(x_1)\cap N_P[x_{m}]$.
Then $x_1Px_ix_mPx_{i+1}x_1$ is a cycle of length $m$ in $G$, a contradiction.
Therefore, we have $N_P^-(x_1)\cap N_P[x_m]=\emptyset$.
Similarly, we have $N_P^+(x_m)\cap N_P[x_1]=\emptyset$.
\end{proof}

Let $N_H^-(x_{1})=N_P^-(x_{1})\cap V(H)$ and   $N_H^+(x_{m})=N_P^+(x_{m})\cap V(H)$.

\begin{claim}\label{3lengthofcp}
Let $P=x_1x_2\cdots x_m$ be a crossing path in $\mathcal{P}$ and $(i,j)$ be a minimal crossing pair of $P$. Let $$U_i= N_H[x_1]\cup( N_H^+(x_{m})\setminus \{x_{i+1}\} )\mbox{ and }V_j= N_H[x_m]\cup( N_H^-(x_{1})\setminus \{x_{j-1}\}).$$
Then the following properties hold.
\begin{itemize}
    \item [$(\romannumeral1)$] $U_i\subseteq V(x_1Px_i)\cup V(x_jPx_m)$ and $V_j\subseteq V(x_1Px_i)\cup V(x_jPx_m)$,
\item [$(\romannumeral2)$] $m-k< j-i-1\leq m-2\ell$, i.e., $2\ell\leq |V(x_1Px_i)\cup V(x_jPx_m)|\leq 2\ell+1$, and
 \item [$(\romannumeral3)$] $|V(x_1Px_i)\cup V(x_jPx_{m})\setminus U_i|=|V(x_1Px_i)\cup V(x_jPx_{m})\setminus V_j|\leq 1$. In particular, if $k$ is odd or $d_H(x_1)+d_H(x_m)=2\ell+1$ or $j-i-1=m-2\ell$, then $V(x_1Px_i)\cup V(x_jPx_{m})=U_i=V_j$.
\end{itemize}
\end{claim}

\begin{proof}
By definition of a minimal crossing pair, we can easily obtain $(\romannumeral1)$.
Since $c(G)<k$ and $x_1Px_ix_mPx_jx_1$ is a cycle of length $m-(j-i-1)$, we have that $m-k < j-i-1$.
Suppose that $j-i-1 > m-2\ell$, i.e., $|V(x_1Px_i)\cup V(x_jPx_m)|<2\ell$.
Since $|N_P^-(x_1)\setminus\{x_{j-1}\}|\geq \ell-1$ and $| N_P[x_{m}]|\geq \ell+1$,
it follows that $N_P^-(x_1)\cap N_P(x_{m})\neq\emptyset$, a contradiction.
Therefore, we have $j-i-1\leq m-2\ell$, proving $(\romannumeral2)$.
Lastly, $(\romannumeral3)$ should follow from the fact that $|U_i|=|V_j|\geq 2\ell$ easily.
\end{proof}

Next we consider the neighbors of end-vertices of a path with a crossing pair.
The following claim strengthens Claim~\ref{3basiclemma}$(\romannumeral3)$ and Claim~\ref{3lengthofcp} and will be used many times throughout the proof.

Let $N^{+1}_P(x_m)= N^{+}_P(x_m)$ and $N^{+i}_P(x_m)= (N^{+(i-1)}_P(x_m))^+$ for $i\geq 2$.

\begin{claim}\label{3structure}
Let $P=x_1x_2\cdots x_m$ be a crossing path with $d_P(x_1)\geq \ell$ and $d_P(x_m)\geq\ell$. If $k$ is even, then $x_1$ is adjacent to all vertices but at most one in $V(P)\setminus (\bigcup^{m-k+1}_{i=1}N^{+i}_P(x_m)\cup \{x_1\})$. Moreover, if $d_P(x_1)=|V(P)\setminus (\bigcup^{m-k+1}_{i=1}N^{+i}_P(x_m)\cup \{x_1\})|$, then $x_1$ is adjacent to each vertex of $V(P)\setminus (\bigcup^{m-k+1}_{i=1}N^{+i}_P(x_m)\cup \{x_1\})$. In particular, if $k$ is odd, then $x_1$ is adjacent to each vertex of $V(P)\setminus (\bigcup^{m-k+1}_{i=1}N^{+i}_P(x_m)\cup \{x_1\})$.
\end{claim}
\begin{proof}
Clearly, since $c(G)<k$, we have $N_P[x_1] \subseteq V(P)\setminus \bigcup^{m-k+1}_{i=1}N^{+i}_P(x_m)$.
Since $P$ has a crossing pair, we have $|V(P)\setminus \bigcup^{m-k+1}_{i=1}N^{+i}_P(x_m)|\leq m-(d_P(x_m)-1)-(m-k+1)=k-d_P(x_m)$.
If $d_P(x_1)\geq \ell$, $d_P(x_m)\geq\ell$ and $k$ is even, then $|V(P)\setminus \bigcup^{m-k+1}_{i=1}N^{+i}_P(x_m)|\leq \ell+2$.
Hence, using $d_P(x_1)\geq \ell$, $x_1$ must be adjacent to all vertices but at most one in $V(P)\setminus (\bigcup^{m-k+1}_{i=1}N^{+i}_P(x_m)\cup \{x_1\})$.
The proof for the rest of Claim~\ref{3structure} is similar and omitted.
\end{proof}

For $P=x_1x_2\cdots x_m\in \mathcal{P}$, let $s_P=\min\{h:x_{h+1}\in N_P(x_m)\}$ and $t_P=\max\{h:x_{h-1}\in N_P(x_1)\}$.\footnote{When there is no ambiguity, we often omit the subscript index in $s_P$ and $t_P$ (such as in the coming claim).}

\begin{claim}\label{3key}
Let $P=x_1x_2\cdots x_m$ be a crossing path in $\mathcal{P}$ with a minimal crossing pair $(i,j)$.
If $x_{s}\in V(H)$ and $x_{s+1}\in N_P(x_1)$, then $x_1$ cannot be adjacent to two consecutive vertices of $x_jPx_{t-1}$.
Similarly, if $x_{t}\in V(H)$ and $x_{t-1}\in N_P(x_m)$, then $x_m$ cannot be adjacent to two consecutive vertices of $x_{s+1}Px_i$.
\end{claim}

\begin{proof}
By symmetry between $x_1$ and $x_m$, we will prove the first statement.
Suppose to the contrary that $x_1$ is adjacent to $x_q$ and $x_{q+1}$ for some $j\leq q\leq t-2$.
We consider the path $R=x_sPx_1x_{s+1}Px_m$.
It follows from $x_s, x_m\in H$ that $R\in\mathcal{P}$.
By the maximality of $m$, we have $N_{H}[x_s]\subseteq V(R)$ and $N_{H}[x_m]\subseteq V(x_{s+1}Rx_m)$.
$R$ has a crossing pair, as otherwise we have $|V(x_1Px_s)|\geq|N_R[x_s]|\geq \ell+1$ and hence $x_1Px_ix_mPx_jx_1$ is a cycle of length at least $|V(x_1Px_{s+1})|+|N^+_R(x_m)\setminus\{x_{i+1}\}|+|\{x_q,x_{q+1}\}|=\ell+1+\ell-1+2\geq k$, a contradiction.
Note that $x_q,x_{q+1}\in  N_P[x_1] \subseteq V(P)\setminus \bigcup^{\theta}_{i=1}N^{+i}_P(x_m)$, where $ \theta=m-k+1$.
Thus by  Claim \ref{3structure}, $x_s$ must be adjacent to one of $x_q,x_{q+1}$.
Suppose that $x_s$ is adjacent to $x_q$.
Then $x_sx_qPx_{s+1}x_mPx_{q+1}x_1Px_s$ is a cycle of length $m$, a contradiction.
Therefore, $x_s$ is adjacent to $x_{q+1}$.
Then $x_sx_{q+1}Px_mx_{s+1}Px_qx_1Px_s$ is a cycle of length $m$, a contradiction.
\end{proof}

Now according to the parity of $k$, we divide the remaining proof into two subsections.
First, we consider the odd case, whose proof is comparably easier, yet revealing essential ideas of our arguments.

\subsection{$k$ is odd.}

In this subsection, we have $k=2\ell+1$.
By Claims \ref{3basiclemma} and \ref{3lengthofcp}, $N_H(x_1)=N_P(x_1)$ and $N_H(x_m)=N_P(x_m)$.

\begin{claim}\label{3oddcrossing}
There exists a crossing path in $\mathcal{P}$.
\end{claim}

\begin{proof}
Suppose to the contrary that all paths in $\mathcal{P}$ are non-crossing.
Let $P=x_1x_2\cdots x_m\in \mathcal{P}$.
Let $\alpha$ be the maximum integer such that $x_{\alpha}$ is adjacent to $x_1$ and $\beta$ be the minimum integer such that $x_{\beta}$ is adjacent to $x_m$.
Note that $\alpha\leq \beta$.
By Lemma \ref{posa lemma}, $G$ contains a cycle of length at least $\min\{m,2\ell+1\}\geq k$, a contradiction.
\end{proof}

By Claim \ref{3oddcrossing}, there is a crossing path $P\in \mathcal{P}$.
Let $(i_1,j_1)$ and $(i_2,j_2)$ be two minimal crossing pairs of $P$ such that $i_1$ is as small as possible and $j_2$ is as large as possible.\footnote{Note that it is possible that $(i_1,j_1)=(i_2,j_2)$.}

\begin{claim}\label{3oddunique}
$P$ has a unique minimal crossing pair $(i,j)$ with $j-i-1=m-k+1$ when $m\geq k+1$.
Moreover, if $m=k$, then each minimal crossing pair $(i^\prime,j^\prime)$ in $P$ satisfies that $j^\prime-i^\prime-1=1$.
\end{claim}

\begin{proof}
Assume that $m\geq k+1$.
Suppose to the contrary that there exist two minimal crossing pairs in $P$, say $i_1<j_1 \leq i_2<j_2$.
By Claim \ref{3lengthofcp}$(\romannumeral2)$, we have that $j_1-i_1-1\geq 2$ and $j_2-i_2-1\geq 2$.
Since $V(x_{i_2+1}Px_{j_2-2})\cap ((N_P^-(x_1)\setminus\{x_{j_1-1}\})\cup N_P[x_{m}])=\emptyset$,
it follows that $x_1Px_{i_1}x_mPx_{j_1}x_1$ is a cycle of length at least $k$, a contradiction.
Let $m=k$.
Then by Claim \ref{3lengthofcp}$(\romannumeral2)$ again, each minimal crossing pair $(i^\prime,j^\prime)$ satisfies that $j^\prime-i^\prime-1=1$.
\end{proof}

\begin{claim}\label{3oddst}
$i_1=s+1$ and $j_2=t-1$.
\end{claim}

\begin{proof}
Let $(i,j)$ be a minimal crossing pair of $P$.
We may assume that $j<t-1$, since otherwise $j_2=t-1$.
By the choices of $s,t$, we have $x_{s+1}\in N_P(x_m)$ and $x_{t-1}\in N_P(x_1)$.
Since $k$ is odd, Claim \ref{3lengthofcp}$(\romannumeral3)$ gives us $x_s,x_{s+1}\in N_H(x_1)$.
Thus it follows from  Claim \ref{3key} that $x_1$ is not adjacent to $x_{t-2}$.
By Claim \ref{3lengthofcp}$(\romannumeral3)$ again, $x_m$ is adjacent to $x_{t-3}$.
Therefore, $(t-3,t-1)$ is a minimal crossing pair in $P$.
So $j_2=t-1$.
Similarly, we have $i_1=s+1$.
\end{proof}

Now we are ready to finish the proof of Lemma \ref{extend posa lemma} when $k$ is odd.
By Claim \ref{3lengthofcp}$(\romannumeral3)$, we have $V(x_1Px_s)\subseteq N_H[x_1]$ and $V(x_tPx_m)\subseteq  N_H[x_m]$.
In particular, this shows $x_s,x_t \in V(H)$.
Note that $x_{s+1}\in N_H(x_m)$ and $x_{t-1}\in N_H(x_1)$.
By Claim \ref{3key}, $x_1$ is not adjacent to consecutive vertices of $V(x_{j_1}Px_{t-1})$ and $ x_m$ is not adjacent to consecutive vertices of $V(x_{s+1}Px_{i_2})$.
By Claim~\ref{3oddst}, we derive that $i_1=s+1$ and $j_2=t-1$.
Let $A=V(x_1Px_s),\ B=V(x_tPx_m),\ C=\{x_{s+1},x_{s+3},\cdots,x_{t-3},x_{t-1}\}$ (if $(i_1,j_1)=(i_2,j_2)$, then $C=\{x_{s+1},x_{t-1}\}$) and $D=V(P)\setminus(A\cup B\cup C)$.
Combining the above arguments with Claims~\ref{3lengthofcp}$(\romannumeral3)$ and~\ref{3oddunique}, we have $N_H[x_1]=A\cup C$ and $N_H[x_m]= B\cup C$.
Consider the crossing paths $R_{\gamma}=x_{\gamma}Px_1x_{\gamma+1}Px_m$ for $2\leq \gamma\leq s$, it is clear that $R_{\gamma}\in \mathcal{P}$ is a crossing path.
By Claim~\ref{3structure}, the neighbors of $x_{\gamma}$ in $H$ are determined by the neighbors of $x_{m}$ in $R$, that is $N_H[x_\gamma]=N_H[x_1]$.
Similarly, $N_H[x_\lambda] =N_H[x_m]$ for $t\leq \lambda\leq m$.
Now it is straightforward to check that $G[V(P)]$ gives a copy in $\mathcal{F}(m,k,s)$ of Type \1, a contradiction.
This completes the proof of Lemma~\ref{extend posa lemma} for odd $k$.

\subsection{$k$ is even.}

In this subsection, we have $k=2\ell+2$.

\begin{claim}\label{3evencrossing}
There exists a crossing path in $\mathcal{P}$.
\end{claim}

\begin{proof}
Suppose to the contrary that all paths in $\mathcal{P}$ are non-crossing.
Let $P=x_1x_2\cdots x_m\in \mathcal{P}$.
Let $\alpha$ be the maximum integer such that $x_{\alpha}$ is adjacent to $x_1$ and $\beta$ be the minimum integer such that $x_{\beta}$ is adjacent to $x_m$.
Note that $\alpha\leq \beta$.

If $\alpha<\beta$, then by Lemma \ref{posa lemma}, $G$ contains a cycle of length at least $\min\{m,2\ell+2\}\geq k$, a contradiction.
Therefore, $\alpha=\beta$.
Since $G$ is 2-connected, there exists a path $Q$ in $G$ with $V(Q)\cap V(P)=\{x_u,x_v\}$ for $1\leq u<\alpha<v\leq m$.
Let $p=\min\{h:h>u,x_h\in N_P(x_1)\}$ and $q=\max\{h:h<v,x_h\in N_P(x_{m})\}.$
Then $C_0=x_1Px_uQx_vPx_{m}x_qPx_px_1$ is a cycle containing $N_P[x_1]\cup N_P[x_{m}]$.
By Claim \ref{3basiclemma}, $C_0$ has length at least $k-1$.
Note that $c(G)<k$.
This forces that $C_0$ has length $k-1$.
It follows that $d_H(x_1)=d_H(x_m)=\ell,\ N_H(x_1)=V(x_2Px_u)\cup V(x_pPx_{\alpha}),\ N_H(x_m)=V(x_{\alpha}Px_q)\cup V(x_vPx_{m-1}),\ V(C)=N_H[x_1]\cup N_H[x_{m}],$ and $Q=x_ux_v$.

For any $2\leq \gamma\leq u-1$, we consider the path $R_{\gamma}=x_{\gamma}Px_1x_{\gamma+1}Px_m$.
Clearly, $R_{\gamma}\in \mathcal{P}$.
Also, by our assumption, $R_{\gamma}$ is non-crossing.
It follows that $N_H[x_{\gamma}]\subseteq V(x_1Px_{\alpha})$.
Suppose that $x_{\gamma}$ has a neighbor $y$ in $V(x_{u+1}Px_{p-1})$.
Then $x_{\gamma}Px_1x_{\gamma+1}Px_uQx_vPx_mx_qPyx_{\gamma}$ is a cycle of length at least $k+1$, a contradiction.

Therefore, we have that $N_H[x_{\gamma}]= N_H[x_1]$ for any $2\leq \gamma\leq u-1$.
Suppose that $p<\alpha$ or $q>\alpha$.
By symmetry, we may assume that $p<\alpha$.
Then we have that $x_{\alpha-1}\in N_P(x_1)$.
Now, we consider the path $L=x_uPx_{\alpha-1}x_1Px_{u-1}x_{\alpha}Px_m$.
Clearly, $L\in\mathcal{P}$.
Note that $x_v\in N_{L}(x_u),\ x_{\alpha}\in N_{L}(x_m)$ and $x_{\alpha}$ precedes $x_v$ in $L$.
It follows that $L$ is a crossing path in $\mathcal{P}$, a contradiction.

The last paragraph implies that $p=\alpha$ and $q=\alpha$.
Suppose that $u=\alpha-1$ or $v=\alpha+1$.
By symmetry, we may assume that $x_{\alpha}x_u\in E(P)$.
Now we consider the path $M=x_uPx_1x_{\alpha}Px_m$.
Clearly, $M\in\mathcal{P}$.
Note that $x_v\in N_{M}(x_u),\ x_{\alpha}\in N_{M}(x_m)$ and $x_{\alpha}$ precedes $x_v$ in $M$.
It follows that $M$ is a crossing path in $\mathcal{P}$, a contradiction.

Thus, we may suppose that $u<\alpha-1$ and $v>\alpha+1$.
Let $A=V(x_1Px_u),\ B=V(x_vPx_m)$ and $C=V(P)\setminus(A\cup B)$.
It is easy to check that $G[V(P)]$ gives a copy in $\mathcal{F}(m,k,\ell)$ of Type \4 (with $k=2\ell+2$, $w=x_\alpha$ and $\{w_1,w_2\}=\{x_u,x_v\}$), a contradiction.
\end{proof}

Let $P\in \mathcal{P}$ be a longest $H$-path with as many minimal crossing pairs as possible and subject to this, let $(i,j)$ be a minimal crossing pair of $P$ with largest length.
Let $(i_1,j_1)$ and $(i_2,j_2)$ be two minimal crossing pairs of $P$ such that $i_1$ is as small as possible and $j_2$ is as large as possible.

\begin{claim}\label{3evenunique}
There is a unique minimal crossing pair in $P$ when $m\geq k+2$. In particular, there are at most two minimal crossing pairs in $P$ when $m=k+1$. Moreover, for $m=k$, each minimal crossing pair $(i^\prime,j^\prime)\neq (i,j)$ in $P$ satisfies $j^\prime-i^\prime=2$.
\end{claim}

\begin{proof}
Assume that $m\geq k+2$.
Suppose to the contrary that there exist two minimal crossing pairs in $P$, that is $i_1<j_1 \leq i_2<j_2$.
By Claim \ref{3lengthofcp}$(\romannumeral2)$, we have that $j_1-i_1-1\geq 3$ and $j_2-i_2-1\geq 3$.
Note that $V(x_{i_2+1}Px_{j_2-2})\cap ((N_P^-(x_1)\setminus\{x_{j_1-1}\})\cup N_P[x_{m}])=\emptyset$.
It follows that $x_1Px_{i_1}x_mPx_{j_1}x_1$ is a cycle of length at least $k$, a contradiction.

Assume that $m=k+1$.
Suppose to the contrary that there exist three minimal crossing pairs $(\alpha_1,\beta_1),(\alpha_2,\beta_2)$ and $(\alpha_3,\beta_3)$ in $P$.
Without loss of generality, we may assume that $\alpha_1<\beta_1 \leq \alpha_2<\beta_2 \leq \alpha_3<\beta_3$.
Note that $(V(x_{\alpha_2+1}Px_{\beta_2-2})\cup V(x_{\alpha_3+1}Px_{\beta_3-2}))\cap ((N_P^-(x_1)\setminus\{x_{\alpha_1-1}\})\cup N_P[x_{m}])=\emptyset$.
Then $x_1Px_{\alpha_1}x_mPx_{\beta_1}x_1$ is a cycle of length at least $k$, a contradiction.

Therefore, $m=k$.
By Claim \ref{3lengthofcp}$(\romannumeral2)$, we have that $j-i-1=1$ or $2$.
We may assume that $j-i-1=2$, since otherwise the result follows by the choice of $(i,j)$.
Hence Claim \ref{3lengthofcp}$(\romannumeral3)$ implies $V(x_1Px_i)\cup V(x_jPx_{m})=U_i=V_j$.
Suppose to the contrary that there exists a minimal crossing pair $(i',j')$ other than $(i,j)$ in $P$ with $j'-i'-1=2$.
It is clearly that $V(x_{i'+1}Px_{j'-2})\cap ((N_P^-(x_1)\setminus\{x_{j-1}\})\cup N_P[x_{m}])=\emptyset$, contradicting $V(x_1Px_i)\cup V(x_jPx_{m})=U_i=V_j$.
\end{proof}

There are two possibilities for the size of $m$: $m\geq k+1$ or $m=k$.
We now split the rest of the proof into two cases based on these two possibilities.

\subsubsection{$m\geq k+1$.}\label{mgeqk+1}

\begin{claim}\label{3evenst}
$i_1=s+1$ and $j_2=t-1$.
\end{claim}

\begin{proof}
By Claim \ref{3evenunique}, there are at most two minimal crossing pairs in $P$.
Suppose that there are two minimal crossing pairs $(i_1,j_1),(i_2,j_2)$ in $P$.
By Claim \ref{3evenunique} again, we have that $m=k+1$.
By Claims \ref{3basiclemma}, \ref{3lengthofcp} and the choices of $s,t$, we have that $j_1-i_1-1=j_2-i_2-1=2$, $x_s,x_{s+1}\in N_H(x_1)$ and $x_{t-1},x_t\in N_H(x_m)$.
Assume that $j_2<t-1$.
Note that $x_{s}\in H$.
It follows from Claim \ref{3key} that $x_1$ is not adjacent to $x_{t-2}$.
By Claim \ref{3lengthofcp}, we have that $x_m$ is adjacent to $x_{t-3}$.
Therefore, $(t-3,t-1)$ is a minimal crossing pair in $P$, contradicting that there are two minimal crossing pairs.
Hence, we have that $j_2=t-1$.
Similarly, we have that $i_1=s+1$.

Thus, we may assume that there is a unique minimal crossing pair $(i,j)$ in $P$.
We will show that $i=s+1$ and $j=t-1$.
Suppose that $d_H(x_1)+d_H(x_m)\geq 2\ell+1$ or $m-(j-i-1)=2\ell$.
Then, similarly as the case when $k$ is odd, Claims~\ref{3basiclemma}$(\romannumeral3)$ and \ref{3key} imply that $i=s+1$ and $j=t-1$.

Now, suppose that $m-(j-i-1)=2\ell+1$ and $d_H(x_1)+d_H(x_m)=2\ell$.
Then there exists a vertex $x_p\in V(x_1Px_i)\cup V(x_jPx_m)$ such that $x_p\notin N_H(x_1)\cup N_H^+(x_m)$.
By symmetry between $x_1$ and $x_m$, we may assume that $1\leq p\leq i$.
Suppose to the contrary that $i>s+1$.
Since $x_t,x_{t-1}\in N_H(x_m)$, by Claim \ref{3key}, we have that $x_{i-1}\notin N_P(x_m)$.
Note that there is only one minimal crossing pair in $P$.
It follows that $V(x_{s+2}Px_{i})\cap N_P(x_1)=\emptyset$.
This forces that $x_{i-2}\in N_P(x_m)$, $s=i-3$ and $p=i$.
Hence, $x_s, x_{s+1}\in N_H(x_1)$.
By Claim \ref{3key} and $p= i$, we can easily deduce that $V(x_jPx_{m-1})\subseteq N_H(x_m)$.
We consider the path $R_1=x_tPx_mx_{t-1}Px_1$.
Clearly, $R_1\in \mathcal{P}$ is a crossing path, and hence by Claim \ref{3structure}, $x_t$ must be adjacent to one of $x_{i-2},x_{i-1}$ (as in the proof of Claim~\ref{3key}).
Then $x_\gamma Px_1x_{t-1}Px_ix_mPx_{t}x_\gamma$ is a cycle of length at least $m-1\geq k$, where $\gamma\in\{i-2,i-1\}$, a contradiction.
This contradiction shows that $i=s+1$.

Next, we will show that $j=t-1$.
First, we suppose that $x_{j+1}\in N_P(x_1)$.
By Claim \ref{3key}, we have that $p=i-1$, or $x_1$ is not adjacent to $x_i$, i.e., $p=i$.
$(a)$ $x_1$ is not adjacent to $x_i$.
It follows that $V(x_1Px_{i-1})\subseteq N_H[x_1]$.
If $i\leq 3$, then $x_1x_jPx_ix_mPx_{j+1}x_1$ is a cycle of length $m-1\geq k$ vertices, a contradiction.
Therefore, $i\geq 4$.
Then we consider the path $R_3=x_{i-2}Px_1x_{i-1}Px_m$.
Clearly, $R_3\in \mathcal{P}$ is a crossing path, and hence by Claim  \ref{3structure}, $x_{i-2}$ must be adjacent to one of $\{x_j,x_{j+1}\}$ (as in the proof of Claim~\ref{3key}).
If $x_{i-2}$ is adjacent to $x_j$, then $x_{1}Px_{i-2}x_{j}Px_ix_mPx_{j+1}x_{1}$ is a cycle of length $m-1\geq k$, a contradiction.
Similarly, if $x_{i-2}$ is adjacent to $x_{j+1}$, then $x_{1}Px_{i-2}x_{j+1}Px_mx_iPx_{j}x_{1}$ is a cycle of length $m-1\geq k$, a contradiction.
$(b)$ $p=i-1$.
Clearly, we have $x_{i-2}\in N_H(x_1)$.
Suppose that $x_{i-2}$ has a neighbor $y\in V(H)$ not in $P$.
Then we consider the path $R_4=yx_{i-2}Px_1x_iPx_m$.
Clearly, $R_4\in \mathcal{P}$ is a crossing path.
Therefore, by Claim~\ref{3structure}, $y$ must be adjacent to one of $\{x_j,x_{j+1}\}$ (as in the proof of Claim~\ref{3key}).
If $y$ is adjacent to $x_j$ (or $x_{j+1}$), then $x_1Px_{i-2}yx_jPx_ix_mPx_{j+1}x_1$ (or $x_1Px_{i-2}yx_{j+1}Px_mx_iPx_jx_1$) is a cycle of length at least $k$, a contradiction.
Therefore, we have $N_H(x_{i-2})\subseteq V(P)$.
Then we consider the path $R_5=x_{i-2}Px_1x_iPx_m$.
\footnote{Note that $R_5$ has $m-1$ vertices. Thus $R_5\notin \mathcal{P}$.}
Obviously, $R_5$ has a crossing pair.
Moreover, it is easy to check that $d_{R_5}(x_{i-2})\geq \ell$ and $d_{R_5}(x_{m})\geq \ell$.
Therefore,  by Claim~\ref{3structure}, $x_{i-2}$ must be adjacent to one of $\{x_j,x_{j+1}\}$.
If $x_{i-2}$ is adjacent to $x_j$ (or $x_{j+1}$), then $x_1Px_{i-2}x_jPx_ix_mP_x{j+1}x_1$ (or $x_1Px_{i-2}x_{j+1}Px_mx_iPx_jx_1$) is a cycle of length $m-1\geq k$, a contradiction.
Thus, by $(a)$ and $(b)$, $x_1$ is not adjacent to $x_{j+1}$.
By Claim \ref{3lengthofcp}$(\romannumeral3)$, $x_m$ is adjacent to $x_j$.
Since there is only one minimal crossing pair, $x_1$ is not adjacent to any vertex of $V(x_{j+1}Px_m)$, that is, $j=t-1$.
This completes the proof of the claim.
\end{proof}

By Claim \ref{3evenunique}, there are at most two minimal crossing pairs in $P$.
Suppose that there are two minimal crossing pairs $(i_1,j_1)$ and $(i_2,j_2)$ in $P$ with $i_1<j_1\leq i_2<j_2$.
From Claim \ref{3lengthofcp}, we deduce $j_1-i_1-1=j_2-i_2-1=2$.
Applying Claim \ref{3evenunique} again, we have $m=k+1$.
By the choices of $s$ and $t$, $x_s,x_{s+1}\in N_H(x_1)$ and $x_t,x_{t-1}\in N_H(x_m)$.
Combining Claim~\ref{3lengthofcp}$(\romannumeral3)$ and \ref{3key}, it is not hard to show that  $j_1=i_2$.
Considering the paths $x_{\gamma}Px_1x_{\gamma+1}Px_m$ and  $x_{\lambda}Px_{k+1}x_{\lambda-1}Px_1$, by Claim~\ref{3structure}, we can determine the neighbors of $x_{\gamma}$ and $x_{\lambda}$ in $H$, that is $N_H[x_1]=N_H[x_{\gamma}]$ for $2\leq \gamma\leq \ell-2$ and $N_H[x_{k+1}]=N_H[x_{\lambda}]$ for $\ell+6\leq \lambda\leq k$.
Let $A=V(x_1Px_{\ell-2}),\ B=V(x_{\ell+6}Px_{k+1}),\ C=\{x_{\ell-1},x_{\ell+2},x_{\ell+5}\}$ and $D=V(P)\setminus(A\cup B\cup C)$.
It is easy to check that $G[V(P)]$ gives a copy in $F_4(k+1,k,\ell-2)$, a contradiction.

Therefore, there is only one minimal crossing pair $(i,j)$ in $P$, that is $i=s+1$ and $j=t-1$ by Claim \ref{3evenst}.
It follows from  Claim \ref{3lengthofcp}$(\romannumeral2)$  that $m-(j-i-1)=2\ell$ or $2\ell+1$.
Suppose that $m-(j-i-1)=2\ell$.
Then applying Claim~\ref{3structure}  as the last paragraph, it is not hard to show that $N_H[x_{\gamma}]=N_H[x_1]$ and $N_H[x_m]=N_H[x_{\lambda}]$ for $2\leq \gamma\leq \ell-1,\  m-\ell+2\leq \lambda\leq m-1$.
In fact, consider the crossing path $x_2x_1x_3Px_m\in \mathcal{P}$, by Claim~\ref{3structure}, $x_2$ is not adjacent to at most one of $\{x_{j-1}\}\cup N_H(x_1)$.
If $x_2$ is  adjacent to $x_{j-1}$, then $x_1x_{i-1}Px_2x_{j-1}Px_ix_mPx_jx_1$ is a cycle of length $m\geq k+1$, a contradiction.
Thus, we have $N_H[x_2]=N_H[x_1]$.
Progressively and similarly, we can show that $N_H[x_{\gamma}]=N_H[x_1]$ and $N_H[x_m]=N_H[x_{\lambda}]$.
Hence, it is easy to check that $G[V(P)]$ gives a copy in $ \mathcal{F}(m,k,\ell-1)$ with Type \2, a contradiction .

Therefore, $m-(j-i-1)=2\ell+1$.
Suppose that $d_H(x_1)+d_H(x_m)=2\ell+1$.
Without loss of generality, let $d_H(x_1)=\ell+1$ and $d_H(x_m)=\ell$.
Then applying Claim~\ref{3structure} as the before, it is not hard to show that $N_H[x_\gamma]\subseteq N_H[x_1]$ for $2\leq \gamma\leq \ell$ and $N_H[x_{m}]=N_H[x_{\lambda}]$ for $\ m-\ell+2\leq \lambda\leq m-1$.
$G[V(P)]$ gives a copy in $\mathcal{F}(m,k,\ell-1)$ with Type \2 (note that there is a Hamilton path starting from $x_1$ and ending at $x_m$, also $x_i$ is adjacent to $x_{i-1}$ and $x_1$ is adjacent to $x_j$), a contradiction.

Now we may assume that $d_H(x_1)=d_H(x_m)=\ell$.
Without loss of generality, there exists a vertex $x_p\notin N_H(x_1)\cup N_H^+(x_m)$ with $1\leq p\leq i$.
Claim \ref{3evenst} implies that $i=\ell+1 \mbox{ and } j=m-\ell+1$.
Also, note that $N_H[x_m]=\{x_{\ell+1}\}\cup V(x_{m-\ell+1}Px_{m})$ and $N_H[x_1]=\{x_{m-\ell+1}\}\cup (V(x_1Px_{\ell+1})\setminus \{x_p\})$.
Then we consider the path $Q_{\lambda}=x_1Px_{\lambda-1}x_mPx_{\lambda}$, where $m-\ell+2\leq \lambda\leq m-1$.
Clearly, $Q_{\lambda}\in \mathcal{P}$ is a crossing path.
As the previous proofs, we have $N_P[x_{\lambda}]\subseteq N_H[x_m]\cup \{x_{p-1}\}$ by Claim~\ref{3structure}.

\begin{claim}\label{one-end-part}
For each $m-\ell+2\leq \lambda\leq m-1$, we have $N_H[x_{\lambda}]=N_H[x_m]$.
\end{claim}

\begin{proof}
Suppose to the contrary that $x_{\lambda}$ is adjacent to $x_{p-1}$.
First we assume that $p<i$.
Then $x_{p-1} P x_1 x_{p+1} P x_{\lambda-1} x_m P x_{\lambda}x_{p-1}$ is a cycle of length $m-1\geq k$, a contradiction.
Therefore, we have $p=i$.
Suppose that $x_{\lambda}$ is not adjacent to $x_i$.
By Claim~\ref{3structure}, $x_{\lambda}$ is adjacent to $x_{i-1}$.
Note that $x_1$ is adjacent to $x_j$.
It follows that there is a minimal crossing pair of longer length in $Q_{\lambda}$, a contradiction.
Therefore, $x_{\lambda}$ is adjacent to $x_i$.
Then we consider the path $L_{\lambda}=x_1Px_{\lambda}x_mPx_{\lambda+1}$.
Clearly, $L_{\lambda}\in \mathcal{P}$ is a crossing path.
By Claim \ref{3structure}, $x_{\lambda+1}$ must be adjacent to one of $\{x_{i-1},x_i\}$.
By the maximality of $j-i-1$, $x_{\lambda+1}$ is adjacent to $x_i$.
Then $x_{i-1}Px_1x_jPx_ix_{\lambda+1}Px_mx_{j+1}Px_{\lambda}x_{i-1}$ is a cycle of length $m$, a contradiction.
Therefore, $x_{\lambda}$ is not adjacent to $x_{i-1}=x_{p-1}$.
This completes the proof of the claim.
\end{proof}

Suppose that $x_p\notin H$.
Let $2\leq \gamma\leq \ell$ and $\gamma\neq p-1$.
Then we consider the path $M_{\gamma}=x_{\gamma}Px_1x_{\gamma+1}Px_m$.
Clearly, $M_{\gamma}\in \mathcal{P}$ is a crossing path.
Note that $x_p\notin N_H(x_{\gamma})$.
By Claim \ref{3structure}, we have that $N_H[x_1]=N_H[x_{\gamma}]$.
Let $A=\{x_1,\cdots,x_{\ell}\}\setminus \{x_p\},\ B=\{x_{m-\ell+2},\cdots,x_m\},\ C=\{x_i,x_j\}$ and $D=V(P)\setminus(A\cup B\cup C)$.
It is easy to check that $G[V(P)]$ gives a copy of $F_1(m,k,\ell-1)$ or $F_2(m,k,\ell-1)$ of Type \3, a contradiction.

Therefore, we have $x_p \in H$.
Then $p\geq 3$.
Let $A=\{x_1,\cdots,x_{\ell}\}$ and $C=\{x_i,x_j\}$.
Then we consider the path $T_{\gamma}=x_{\gamma}Px_1x_{\gamma+1}Px_m$ for $2\leq \gamma\leq \ell$ with $\gamma\neq p-1$.
Clearly, $T_{\gamma}\in \mathcal{P}$ is a crossing path.
Firstly, let $p=i$.
By Claim~\ref{one-end-part}, we may consider the path $Q_1=x_{p-1}Px_1x_{m-\ell+1}Px_{i}x_{m-\ell+2}Px_m$.
Clearly, $Q_1\in \mathcal{P}$ is a crossing path.
Let $\mathcal{P}_1=\{Q_1\}\cup \{T_{\gamma}: 2\leq \gamma\leq \ell\}$.
Considering each path in $\mathcal{P}_1$,
by Claim \ref{3structure}, we have that each vertex of $A$ in $G[A\cup C]$ has degree at least $\ell$.
Note that both vertices of $C$ are adjacent to $A$.
It is easy to check that $G[V(P)]$ gives a copy of $F\in \mathcal{F}(m,k,\ell-1)$ of Type \2, a contradiction.
Secondly, let $p\leq \ell=i-1$ and $\ell\geq 4 $.
Then we consider the crossing path $Q_2=x_{p}Px_1x_{p+1}Px_m$.
By the maximality of $Q_2$, $N_H(x_p)\subseteq V(Q_2)=V(P)$.
Since $x_p\notin N_H(x_1)$, by Claim \ref{3structure}, we have that $x_p$ is adjacent to each vertex of $\{x_2,\cdots,x_{\ell+1}\}$.
Considering each path of $T_{\gamma}$ with $2\leq\gamma \leq \ell$ and $\gamma\neq p-1$, and the crossing paths $x_{p-1}x_px_{p-2}Px_1x_{p+1}Px_m$ when $p\geq 4$ and $x_{p-1}x_px_{p+1}x_1x_{p+2}Px_m$ when $p=3$, by Claim \ref{3structure}, each vertex of $A$ in $G[A\cup C]$ has degree at least $\ell$.
It is easy to check that $G[V(P)]$ gives a copy of $F\in \mathcal{F}(m,k,\ell-1)$ of Type \2, a contradiction.
Finally, $p\leq \ell$ and $\ell\leq 3$.
Suppose that $m\geq k+2$.
Then $j= m-\ell+1\geq \ell+5\geq 2\ell+2$.
Thus $x_1Px_jx_1$ is a cycle of length at least $2\ell+2=k$, a contradiction.
Now let $m=k+1$, that is $j=m-\ell+1$.
If $\ell\leq 2$, then $x_1Px_jx_1$ is a cycle of length at least $k$, a contradiction.
Let $\ell=3$.
This forces that $k=8$ and $p=3$.
Suppose that $N_P(x_2)\geq 3$.
Then $G$ contains a copy of $F\in \mathcal{F}(9,8,2)$ ($A=\{x_1,x_2,x_3\}$) with Type \2, a contradiction.
Therefore, $N_P(x_2)= 2$.
It follows that there is a vertex $z\in N_H(x_2)\setminus N_P(x_2)$.
Thus $N_H(z)= \{x_2,x_4,x_7\}$.
Let $A=\{x_1\}$, $B=\{z\}$ and $C=\{x_2,x_4,x_7\}$.
Then we consider the path $zx_2Px_7x_1$.
It is easy to check that $G[V(P)]$ gives a copy in $F\in \mathcal{F}(8,8,1)$ with Type \2, a contradiction.
This completes the proof when $k$ is even and $m\geq k+1$.

\subsubsection{$m=k$.}\label{3m=k}

By Claim \ref{3lengthofcp}$(\romannumeral2)$, we have that $j-i-1=1$ or $2$.
Suppose that either $j-i-1=2$ or $j-i-1=1$ and $d_H(x_1)+d_H(x_k)=2\ell+1$.
Then the same proof as in the Subsection~\ref{mgeqk+1} shows that
$G[V(P)]$ gives a copy of $ \mathcal{F}(k,k,s)$ with Type \2, a contradiction.

Therefore, $j-i-1=1$ and $d_H(x_1)=d_H(x_k)=\ell$.
Without loss of generality, there exists a vertex $x_p\notin N_H(x_1)\cup N_H^+(x_k)$ with $1\leq p\leq i$.
By the definition of $t$, we have that $i\leq t-2$.
Now, subject to previous choices, we choose $P\in \mathcal{P}$ such that $|V(x_{s+2}Px_{t-2})\cap \{x_p\}|$ is as large as possible.

\begin{claim}\label{3even3C1}
$p\leq s+1$.
\end{claim}

\begin{proof}
Suppose to the contrary that $p>s+1$, that is $s+2\leq p\leq t-2$.
Then $x_s$ and  $x_{s+1}$ belong to $N_H(x_1)$, and $x_{t-1}$ and $x_t$ belong to $N_H(x_k)$.
Clearly, we have that $x_p\notin N_H(x_1)$ and $x_{p-1}\notin N_H(x_k)$.

Suppose that $x_{p-1}\in N_H(x_1)$ and $x_p\in N_H(x_k)$.
Let
\begin{equation}
C=\{x_{s+1}, x_{s+3}, \cdots, x_{p-5}, x_{p-3}, x_{p+2}, x_{p+4}, \cdots, x_{t-3}, x_{t-1}\}.\nonumber
\end{equation}
We shall show that $x_1$ and $x_k$ are adjacent to each vertex of $C$.
By Claim \ref{3lengthofcp}$(\romannumeral3)$, $p> s+1$ and definition of $s$, $x_1$ and $x_k$ are adjacent to $x_{s+1}$.
By Claim \ref{3basiclemma}$(\romannumeral3)$ and Claim \ref{3key}, $x_1$ and $x_k$ are not adjacent to $x_{s+2}$.
Next, by Claim \ref{3lengthofcp}$(\romannumeral3)$ and  $x_1$ is adjacent to $x_{s+3}$ with $p>s+3$ and hence by Claim \ref{3key}, $x_1$ is not adjacent to $x_{s+4}$.
Then it follows from Claim \ref{3lengthofcp}$(\romannumeral3)$ and \ref{3key} that $x_k$ is adjacent to $x_{s+3}$ and not adjacent to $x_{s+4}$ with $p> s+4$.
Progressively, we can show that $x_1$ and $x_k$ are adjacent to each vertex of $\{x_{s+1},x_{s+3},\cdots,x_{p-5},x_{p-3}\}$.
Moreover, we also show that $i_1=s+1$.
Similarly, we have that $x_1$ and $x_k$ are adjacent to each vertex of $\{x_{p+2},x_{p+4},\cdots,x_{t-3},x_{t-1}\}$ and $j_2=t-1$.
Thus, $x_1$ and $x_k$ are adjacent to each vertex of $C$.
Moreover, we have $s\equiv p$ modulo $2$ and $t-1\equiv p$ modulo $2$.
Then we consider the paths $T_{\gamma}=x_{\gamma}Px_1x_{\gamma+1}Px_k$ and $S_{\lambda}=x_{\lambda}Px_kx_{\lambda-1}Px_1$ for $2\leq \gamma\leq s$ and $ t\leq \lambda\leq k-1$.
Clearly, $T_{\gamma},S_{\lambda}\in \mathcal{P}$.
By Claims \ref{3basiclemma}$(\romannumeral3)$, \ref{3lengthofcp}$(\romannumeral3)$, \ref{3key} and maximality of $j-i-1$, we have that $N_H[x_{\gamma}]=N_H[x_1]$ and $N_H[x_{\lambda}]=N_H[x_k]$.
Hence $x_{p-1}x_1Px_sx_{t-1}Px_px_kPx_{t}x_{s+1}Px_{p-1}$ is a cycle of length $k$, a contradiction.

Suppose that $x_{p-1}\in N_H(x_1)$ and $x_p\notin N_H(x_k)$.
Then $p< i$ and $x_{p+1}\in N_H(x_1)$.
Let
\[C=\{x_{s+1},x_{s+3},\cdots,x_{p-1},x_{p+1},\cdots,x_{t-3},x_{t-1}\}.\]
Then the same proof as in the last paragraph shows that $x_1$ is adjacent to each vertex of $C$ and $x_k$ is adjacent to each vertex of $C\setminus\{x_{p-1}\}$.
Let $A=V(x_tPx_k)$ and $B=V(x_1Px_s)$.
Note that $d_H(x_1)=d_H(x_k)=\ell$.
Then we have $|A|=s+1$ and $|B|=s$.
Considering the paths $x_{\lambda}Px_kx_{\lambda-1}Px_1\in \mathcal{P}$ for $t\leq \lambda\leq k-1$,  by Claim  \ref{3basiclemma}$(\romannumeral3)$, we have that $N_H(x_{\lambda})\subseteq A\cup C$.
Moreover, considering the paths $x_{\gamma}Px_1x_{\gamma+1}Px_k\in \mathcal{P}$ for $ 2\leq \gamma\leq s$, similarly as the proof as in the last paragraph, we have that $N_H(x_{\gamma})= B\cup C$.
Note that $x_k$ is adjacent to both end-vertices of $x_{s+1}Px_{t-1}$.
It is easy to check that $G[V(P)]$ gives a copy in $\mathcal{F}(k,k,s)$ with Type \2.

Therefore $x_{p-1}\notin N_H(x_1)$.
If $x_p\notin N_P(x_k)$, then there is a minimal crossing pair $(i^\prime,j^\prime)$ of length at least two ($x_{p-1},x_p\notin N_P(x_k)$ and $x_{p-1},x_p\notin N_P(x_k)$), a contradiction.
Therefore we have that $x_p\in N_H(x_k)$.
Let
\[C=\{x_{s+1},x_{s+3},\cdots,x_{p-2},x_p,x_{p+2},\cdots,x_{t-3},x_{t-1}\}.\]
As the proofs before, $x_k$ is adjacent to each vertex of $C$ and $x_1$ is adjacent to each vertex of $C\setminus\{x_p\}$.
Let $A=V(x_1Px_s)$ and $B=V(x_tPx_k)$.
From $d_H(x_1)=d_H(x_k)=\ell$, we have $|A|=s$ and $|B|=s-1$.
Considering the paths $x_{\gamma}Px_1x_{\gamma+1}Px_k\in \mathcal{P}$ and $x_{\lambda}Px_kx_{\lambda-1}Px_1\in \mathcal{P}$ for $2\leq \gamma\leq s$ and $t\leq \lambda\leq k-1$, as the previous proofs, we have $N_H[x_{\gamma}]\subseteq A\cup C$ and $N_H[x_{\lambda}]= B\cup C$.
Note that $x_1$ is adjacent to both end-vertices of $x_{s+1}Px_{t-1}$.
It is easy to check that $G[V(P)]$ gives a copy in $\mathcal{F}(k,k,s-1)$ with Type \2, a contradiction.
Thus we finish the proof of Claim \ref{3even3C1}.
\end{proof}

The rest proof is similar as the case $m=k+1$.
Let $2\leq p\leq s+1$, that is, $|V(x_{s+2}Px_{t-2})\cap \{x_p\}|=0$.
Then $x_{t}$ and $x_{t-1}$ belong to $N_H(x_k)$ and $p\leq i_1$.
By Claim \ref{3key}, $x_k$ is not adjacent consecutive vertices of $V(x_{s+1}Px_{i_2})$.
First, we will show that $i_1=s+1$.
Suppose that $i_1>s+1$.
Then by Claim \ref{3key}, $x_k$ is not adjacent to $x_{s+2}$, and hence $i_1\geq s+3$.
Also, since $2\leq p\leq s+1$, by Claim \ref{3lengthofcp}$(\romannumeral3)$, $x_1$ is adjacent to $x_{s+3}$, that is $(s+1,s+3)$ is a crossing pair, contradicting $i_1>s+1$.

\begin{claim}
$N_H[x_{\lambda}]=N_H[x_k]$ for $t\leq \lambda\leq k-1$.
\end{claim}

\begin{proof}
By Claim \ref{3structure}, we have $N_H[x_{\lambda}]\subseteq N_H[x_k]\cup \{x_{p-1}\}$.
If $d_H(x_{\lambda})\geq \ell+1$, the we are done as the beginning of the case $m=k$ is even.
Let $d_H(x_{\lambda})= \ell$.
Suppose that $x_{\lambda}$ is adjacent to $x_{p-1}$.
Let $p<s+1$.
Since $x_p\notin N_H^+(x_m)$, $x_1$ is adjacent to $x_{p+1}$.
Considering the path $P^{\lambda}=x_{\lambda}Px_kx_{\lambda-1}Px_1\in \mathcal{P}$, by the maximality of the number of minimal crossing pairs of $P$, $x_{\lambda}$ is adjacent to each vertex of $V(x_{t-1}Px_k)$.
Hence, $x_{\lambda}$ is not adjacent to a vertex $x\in N_H[x_k]\setminus V(x_{t-1}Px_k)$.
Note that $N_{P^{\lambda}}(x_{\lambda})\cup\{x\}= N_P(x_k)\cup \{x_{p-1}\}$.
Thus the path $P^{\lambda}=x_{\lambda}Px_kx_{\lambda-1}Px_1=y_kPy_1$ ($y_k=x_{\lambda}$ and $y_1=x_1$) with a minimal crossing pair $(i^\prime,j^\prime)$ satisfies $|V(y_{s^\prime+2}Py_{t^\prime-2})\cap \{y_{p^\prime}\}|=1$, where $s^\prime=\min\{h:y_{h+1}\in N_{P^\prime}(y_k)\}$, $t^\prime=\max\{h:y_{h-1}\in N_{P^\prime}(y_1)\}$ and $\{y_{p^\prime}\}=(V(y_1Py_{i^\prime})\cup V(y_{j^\prime}Py_k))\setminus (N_{P^{\lambda}}[y_1]\cup N^+_{P^{\lambda}}(y_k))$, a contradiction to the choice of $P$ .
Now, let $p=s+1$.
Note that $i_1=s+1=p$.
Then the same proof as in Claim~\ref{one-end-part} shows that $N_P[x_k]=N_P[x_{\lambda}]$.
This completes the proof of the claim.
\end{proof}

\begin{claim}\label{claim:p=s+1}
$p=s+1$.
\end{claim}

\begin{proof}
Suppose to the contrary that $p\leq s$.
First we assume that $2\leq p\leq s-1$.
Since $\{x_{s},x_{s+1}\}\subseteq N_H(x_1)$, by Claim \ref{3key}, $x_1$ is not adjacent to any two consecutive vertices of $x_{s+1}Px_{t-1}$.
Suppose that $x_p\notin H$.
By Claims \ref{3structure}, we have $N_H[x_\gamma]=N_H[x_1]$ for $2\leq \gamma\leq s$.
Combining the fact that $N_H[x_{\lambda}]=N_H[x_k]$ for $t\leq \lambda\leq k-1$, it is easy to check that $G$ contains either a copy of $F_1(k,k,s)$ or $F_2(k,k,s)$ with Type \3, a contradiction.
Therefore $x_p\in H$.
Then $p\geq 3$ and $s\geq 4$.
Let $A=\{x_1,\cdots,x_s\}$ and $C=\{x_{s+1},x_{s+3},\cdots,x_{t-3},x_{t-1}\}$.
For $x\in A\setminus \{x_{p-1}\}$, we can deduce $N_H[x]\subseteq A\cup C$ from Claim \ref{3structure} as the previous proofs.
By the maximality of $P$, $N_H(x_p)$ belongs to $V(P)$.
It follows that $x_1$ is adjacent to each vertex of $(A\cup C)\setminus \{x_p\}$ and $x_{p}$ is adjacent to each vertex of $(A\cup C)\setminus \{x_1\}$.
Now, considering the path $x_{p-1}x_{p}x_{p-2}Px_1x_{p+1}Px_m\in \mathcal{P}$, Claim \ref{3structure} implies $N_H(x_{p-1})\subseteq A\cup C$.
Note that $x_1$ is adjacent to each vertex of $C$.
Hence, it is easy to check that $G[V(P)]$ gives a copy in $ \mathcal{F}(k,k,s)$ with Type \2, a contradiction.

Therefore we can assume that $p=s$.
Suppose that $x_1$ is not adjacent to any two consecutive vertices of $x_{j_1}Px_{t-1}$.
Then the same proof as in the last paragraph shows that $G[V(P)]$ gives a copy in $ \mathcal{F}(k,k,s)$ with Type \2, a contradiction.
Therefore $x_1$ is adjacent to two consecutive vertices of $x_{j_1}Px_{t-1}$.
Let $\lambda$ be the minimum integer such that $x_1$ is adjacent to both $x_{\lambda}$ and $x_{\lambda+1}$.
By Claim \ref{3key}, we have that $x_s\notin H$.
Let $r=\min\{h: h\geq \lambda, x_h\in N_H(x_k)\}$.
By Claim \ref{3basiclemma} and $p=s<i_1$, we have that $V(x_{\lambda}Px_{r})\subseteq N_H(x_1)$.
Hence, we have that  $x_{r},x_{r-1}\in N_H(x_1)$.
It follows from Claim \ref{3key} that $x_1$ is not adjacent to any two consecutive vertices of $x_{r+2}Px_{t-1}$.
Considering the path $x_{\gamma}Px_1x_{\gamma+1}Px_k\in \mathcal{P}$ for $\gamma\in [2,s-2]\cup[\lambda,r-1]$, by $x_s\notin H$ and Claim~\ref{3structure}, we have $N_P[x_1]=N_P[x_{\gamma}]$.
Let $A=V(x_1Px_{s-1})\cup V(x_{\lambda}Px_{r-1})$, $B=V(x_{t+1}Px_k)$ and $C=\{x_{s+1},x_{s+3},\cdots,x_{\lambda-2},x_{r-1},x_{r},x_{r+2},\cdots,x_{t-2},x_{t}\}$.
Note that $|A|=|B|=k-t$.
Moreover, $G[A]$ is a complete graph.
It is easy to check that $G[V(P)]$ gives a copy in $ \mathcal{F}_3(k,k,k-t)$ with Type \3, a contradiction.
This proves Claim~\ref{claim:p=s+1}.
\end{proof}

Therefore, we have $p=s+1$.
Note that $x_p\in N_H(x_k)$.
Then $x_p\in H$ and $p\geq 3$.
First, we show that $x_1$ is not adjacent to any two consecutive vertices of $V(x_{j_1}Px_t)$.
Suppose to the contrary that $x_1$ is adjacent to both of $\{x_{\lambda},x_{\lambda+1}\}$ for some $\lambda\geq j_1$.
Considering the path $x_{s-1}Px_1x_sPx_k$, it follows from Claim \ref{3structure} that $x_{s-1}$ is adjacent to at least one of $\{x_{\lambda},x_{\lambda+1}\}$.
Hence, $P^\prime_1=x_sPx_{\lambda}x_{s-1}Px_1x_{\lambda+1}Px_k\in \mathcal{P}$ (or $P^\prime_1=x_sPx_{\lambda}x_{1}Px_{s-1}x_{\lambda+1}Px_k\in \mathcal{P}$, we omit the proof of this case) is a crossing path on $k$ vertices ending at $x_k$.
Note that $x_k$ is not adjacent to both of $\{x_1,x_{\lambda-1}\}$.
By Claim \ref{3structure}, $x_s$ is adjacent at least one of $\{x_{\lambda},x_{\lambda+1}\}$.
Then $x_1Px_sx_{\lambda}Px_{s+1}x_kPx_{\lambda+1}x_1$ (or $x_1x_{\lambda}Px_{s+1}x_kPx_{\lambda+1}x_sPx_1$) is a cycle of length $k$, a contradiction.
Therefore $x_1$ is not adjacent to any two consecutive vertices of $x_{j_1}Px_{t-1}$.
Let $A=\{x_1,x_2,\cdots,x_s\}$, $B=\{x_t,x_{t+1},\cdots,x_k\}$ and $C=\{x_{s+1},x_{s+3},\cdots,x_{t-3},x_{t-1}\}$.
Then we have $N_H[x_1]=(A\cup C)\setminus \{x_{s+1}\}$ and $N_H[x_k]=B\cup C$.
Since $N_H[x_k]=N_H[x_{\lambda}]$ for $t\leq \lambda\leq k-1$, by Claim \ref{3structure}, we have that $N_H[x_{\gamma}]\subseteq A\cup C $ for $2\leq \gamma\leq s$.
Note that $x_s$ is adjacent to $x_{s+1}$.
It is easy to check that $G[V(P)]$ gives a copy in $ \mathcal{F}(k,k,s-1)$ with Type \2, a contradiction.
This completes the proof of Lemma~\ref{extend posa lemma}.
\qed

\section{Proof of the main result}\label{main result}
\noindent For a family of graphs $\mathcal{F}$, we say a graph $G$ is {\it $\mathcal{F}$-free} if it does not contain any $F\in \mathcal{F}$  as a subgraph.
Let $F(\ell)$ be the graph obtained by taking a path $P_{2\ell-1}$ on $2\ell-1$ vertices and a disjoint copy of $\overline{K}_3$,
and joining each vertex of $\overline{K}_3$ to each vertex of the larger partite set in the unique bipartition of $P_{2\ell-1}$.

\medskip

Let $k\geq 5$ and $\mathcal{K}_{k,0}=\emptyset$. For $1\leq \alpha\leq \ell-2$, let $\mathcal{K}_{k,\alpha}$ be the family of the following graphs:\footnote{If $k$ is odd, then $\mathcal{K}_{k,\alpha}$ only contains graphs in $\mathcal{F}(m,k,r)$ with  $r\in\{1,\ldots,\alpha,\ell-1\}$.}\\
\indent $(a)$ $F\in \mathcal{F}(m,k,r)$ with  $r\in\{1,\ldots,\alpha\}\cup\{\ell-1,\ell\}$,\\
\indent $(b)$  $F_0(k,k,\ell-2)$ and $F_4(k+1,k,\ell-2)$  when $k\geq 10$ is even and $\ell -\alpha\leq 3$,\\
\indent $(c)$  $F_2(m,k,\alpha+1)$ with $\alpha+1\leq \ell-2$ when $k$ is even,\\
\indent $(d)$  $F_5(m,k,2)$ when $\alpha=1$ and $k$ is even, and\\
\indent $(e)$  $F(\ell)$ when $k$ is even.

\medskip

For a given family of graphs $\mathcal{F}$, we say a graph $G$ is a {\it maximal} $\mathcal{F}$-free graph with $c(G)<k$ if, for any non-edge $ab$ of $G$, $G+ab$ contains either a copy of $F\in\mathcal{F}$  or a cycle of length at least $k$.

The following theorem is the main result of this paper, from which one can derive Theorem~\ref{corollary 2-connected large n}
and some other results (such as the results of \cite{Furedi2016,Furedi2018,MN2019}), to be discussed in Section \ref{corollary}.
Mainly, it says that by forbidding some family $\mathcal{K}_{k,\alpha}$, one can have a good understanding on structural properties of graphs with given circumference and relatively many $s$-cliques.

\begin{theorem}\label{Theorem 2-connected}
Let $k\geq 5$, $\alpha\geq 0$ and $\beta\geq 2$ be integers. Let $G$ be an $n$-vertex $2$-connected maximal $\mathcal{K}_{k,\alpha}$-free graph with $c(G)<k$. If $\ell-\alpha\geq \beta $ and
\begin{equation}\label{bound for 2-connected}
N_s(G)>\max\{h_s(n,k,\ell-\alpha),h_s(n,k,\beta)\},
\end{equation}
then we have either $\omega(G)> k-\beta $ or $|V(H(G,\ell-1))|< k-\ell+\alpha$.
\end{theorem}

We note that if $\alpha$ or $\beta$ is larger, then $\max\{h_s(n,k,\ell-\alpha),h_s(n,k,\beta)\}$ is smaller and presumably the structure of $G$ becomes more complicated.
Also we have $\omega(G)\leq k-2$, so $(b)$ does not occur when $\beta= 2$.
Equivalently, Theorem~\ref{Theorem 2-connected} states that an $n$-vertex 2-connected graph $G$ satisfying (\ref{bound for 2-connected}) with $\beta=2 $ and $|V(H(G,\ell-1))|\geq k-\ell+\alpha$ contains  either  a copy of $K\in \mathcal{K}_{k,\alpha}$ or a cycle of length at least $k$.

\subsection{Some facts on $\mathcal{F}(m,k,r)$ with $r\leq \ell-2$}
We need the following technical propositions. 

\begin{proposition}\label{basic property for A=r+1}
Let $G$ be an $n$-vertex connected graph with a non-edge $c_1c_2$ and $n\geq 6$. Assume that each vertex except $c_1$ and $c_2$ of $G$ has degree $n-2$. Then the following hold:\\
$(\romannumeral1)$ For each $ab\in E(G)$, there is a Hamilton path starting from $c_1$ through $ab$ and ending at $c_2$.\\
$(\romannumeral2)$ For each $v \in V(G)\setminus\{c_1,c_2\}$, there is a  path on $n-1$ vertices starting from $v$ ending at $\{c_1,c_2\}$.\\
$(\romannumeral3)$ For each non-edge $ab\neq c_1c_2$ of $G$, there is a path starting from $c_1$ through $ab$ and ending at $c_2$ on at least $n$ vertices in $G+ab$ except when $\{d_G(c_1),d_G(c_2)\}=\{1,n-3\}$.
\end{proposition}
\begin{proof}
Let $A=V(G)\setminus\{c_1,c_2\}$.
We divide the vertex of $A$ into $A_0$, $A_1$ and $A_2$ such that each vertex in $A_0$ is adjacent to both of $\{c_1,c_2\}$ and each vertex of $A_i$ is not adjacent to of $c_i$ for $i=1,2$.
Since each vertex  of $A$ has degree $n-2$, $G[A_1]$ and $G[A_2]$ are complete graphs and $G[A_0]$ is the complement of the graph of $|A_0|/2$ independent edges.
Moreover, since $G$ is connected, if $|A_0|=0$, then $|A_1|\geq 1$ and $|A_2|\geq 1$.
Note that $|A|\geq 4$.
For each $ab\in E(G)$, we can easily find a Hamilton path starting from $c_1$ through $ab$ and ending at $c_2$ (consider $|A_0|=0$, $|A_0|=2$ and $|A_0|\geq 4$ separately).
Thus we finish the proof of $(\romannumeral1)$.
The proof of $(\romannumeral2)$ is similar.
Now let $ab$ be a non-edge.
Adding the edge $ab$ and deleting at most two independent edges between $\{a,b\}$ and $\{c_1,c_2\}$ such that each vertex of $A$ in the obtained graph has degree $n-2$.
If the obtained graph is connected, then one can easily prove $(\romannumeral3)$ by applying $(\romannumeral1)$.
Assume that the obtained graph is not connected. Then $\{a,b\} \cap \{c_1,c_2\}\neq \emptyset$, and hence it is easy to see that $\{d_G(c_1),d_G(c_2)\}=\{1,n-3\}$.
Moreover, we have $G[A]=K_{n-2}$.
The proof is complete.
\end{proof}


Recall vertices $x, x_i, y, y_1, z_i, z_i'$ in those special graphs in $\mathcal{F}(m,k,r)$.
For $F=F_0(m,k,r)$, we denote by $v$ the isolated vertex in $F[D]$ and $v_1, v_2$ the neighbour of $v$ in $F[C\cup D]$, respectively.
For $F=F_2(m,k,r)$, we denote by $y_2$ the neighbour of $y$ in $F[C\cup D]$.
For $F\in \mathcal{F}_3(m,k,r)$, we denote by $z_2, z^\prime_2$ the neighbour of $z_1, z^\prime_1$ in $F[C\cup D]$, respectively.
For $F=F_5(m,k,2)$ with $F[A]=S_3$, where $S_3$ is a star in three vertices, we denote by $u_1$ the center of $S_3$.

\begin{proposition}\label{property for H}
For $1\leq r\leq \ell-2$, each $F\in\mathcal{F}(m,k,r)$ satisfies the following:
\begin{itemize}
\item [$(\romannumeral1)$]  Let $ab\in E(F)$. If $ab\in\{x_1x_2,z_1z_2,z_1^\prime z_2^\prime,vv_1,vv_2,y_1y_2\}$, $ab\in E(\{u_1\}, C)$ or $ab\in E(\{y_1\}, C)$ and $r\geq 2$,
then there is a cycle of length $k-2$ containing $ab$; otherwise, there is a cycle of length $k-1$ containing $ab$.
\item [$(\romannumeral2)$] For each non-edge $ab$ in $A\cup B\cup D$, if $\{a,b\}\subseteq A$, $\{a,b\}\subseteq A\cup\{x\}$, $\{a,b\}\subseteq A\cup\{y\}$, $u_1\in \{a,b\}$ or $y_1\in \{a,b\}$ and $r\geq 2$,
then $F+ab$ contains a cycle of length $k-1$  containing $ab$; otherwise, $F+ab$ contains a cycle of length at least $k$ containing $ab$.
\item [$(\romannumeral3)$] For each non-edge $ab$ between $ A\cup B\cup D$ and $C$, $F+ab$ contains a cycle of length at least $k-2$  containing $ab$. Moreover, if $ab$ is between $A$ and $C$ with $u_1\notin\{ a,b\}$, then $F+ab$ contains a cycle of length at least $k-1$  containing $ab$.
\item [$(\romannumeral4)$] Suppose that $G$ is $2$-connected with $c(G)<k$ and containing a copy of $F$. Then $G-A\cup B\cup C$ is a star forest.
\end{itemize}
\end{proposition}
\begin{proof} Note that $|C|=\ell-r+1\geq 3$. We only verify some special cases and leave other cases to readers.

$(\romannumeral1)$. For $ab=x_1x_2$, since the longest path starting from $x_1$ ending at $x_2$ contains at most $|D|-2$ vertices of $D$, the result follows. For $|A|=r+1\geq 4$ and $k$ is even, the result follows from Proposition~\ref{basic property for A=r+1}$(\romannumeral1)$. Be careful! For $|A|=r+1=2$ and even $k$, by the definition of $\mathcal{F}(m,k,r)$, we have $F[A]=K_2$.

$(\romannumeral2)$. Let $\{a,b\}\subseteq A\cup\{y\}$ be a non-adjacent pair. Then we have $y_1\in \{a,b\}$ and $r\geq 2$. Since the longest path starting from $y_1$ in $F$ is on at most $k-1$ vertices when $r\geq 2$, the result follows. For $|A|=r+1\geq 4$ and $k$ is even, the result follows from Proposition~\ref{basic property for A=r+1}$(\romannumeral2)$ easily.

$(\romannumeral3)$. Let $F=F_4(k+1,k,3)$ and $\{a,b\}$ be non-adjacent pair between $A\cup B \cup C$ and $D$.
Then the longest path starting from $a$  ending at $b$ contains all vertices of $A\cup B\cup C$ and at least one vertex of $D$.
Thus, we have $c(F+ab)\geq k+1-3=k-2$.
The result follows.
Let $|A|=r+1\geq 4$, $k$ be even and $c_1,c_2$ be the end-vertices of $F[C\cup D]$. Without loss of generality, let $d_{F[A\cup C]}(c_1)\leq d_{F[A\cup C]}(c_2)$.  If $F[A]=K_{\ell}$ and $d_{F[A\cup C]}(c_1)=1$, then it is easy to see that $(\romannumeral3)$ holds (note that $|C|\geq 3$).
Otherwise, the result follows from Proposition~\ref{basic property for A=r+1}$(\romannumeral3)$.

$(\romannumeral4)$. Let $X$ be a non-trivial component of $G-A\cup B\cup C$.\footnote{We say a component is {\it trivial} if it consists of a unique vertex.}
Since $G$ is 2-connected with $c(G)<k$, by $(\romannumeral1)$, $(\romannumeral2)$, and $(\romannumeral3)$, $X$ is only connected to $C$.
Note that, for any two vertices $s_1,s_2\in C$, there is a path on  at least $k-3$ vertices starting from $s_1$ and ending at $s_2$.
The longest path starting from $C$ through $X$ ending at $C$ is on at most four vertices.
Then there is an edge $uv$ in $X$ which is connected to $C$ by two independent edges.
Moreover, $V(X)-\{u,v\}$ is an independent set of $G[X]$ and each vertex of $X-\{u,v\}$ is adjacent to the same vertex of $\{u,v\}$.
Otherwise, it is not hard to show that $G$ contains a cycle of length at least $k$, a contradiction.
Thus $G[X]$ is a star.
This finishes the proof of the proposition.\end{proof}

Let $E_{n-k+1}$ be the $(n-k+1)$-vertex graph consisting of $\lfloor \frac{n-k+1}{2}\rfloor$ independent edges.
Let $G(n,k,3)$ be the graph obtained from a disjoint union of $F_4(k+1,k,\ell-2)$ and $E_{n-k+1}$ by joining each vertex of the set $C$ in $F_4(k+1,k,\ell-2))$ to each vertex in the set $D$ and $V(E_{n-k+1})$.
Denote by $g_s(n,k,3)$ the number of unlabeled $s$-cliques of $G(n,k,3)$. Recall that $h_s(n,k,r)$ is the number of unlabeled $s$-cliques of $H(n,k,r)$. Also recall that $F_4(m,k,\ell-2)$ is the only graph with $m>k$ and $r\leq \ell-2$. We need the following lemma to prove our main theorem.

\begin{lemma}\label{bound the edges of F small s}
Let $G$ be a $2$-connected graph on $n$ vertices with $c(G)<k$. Let $m\geq k\geq 9$ and $1\leq r\leq \ell-2$. Suppose that $G$ contains a copy of $F\in \mathcal{F}(m,k,r)$. Then\\
$(\romannumeral1)$ Let $ \gamma=\min\{\ell-r+2,\ell\}$. If $F=F_2(k,k,r)$, then
\begin{equation}
N_s(G)\leq \min\{h_s(n,k,\gamma),\ldots,h_s(n,k,\ell)\},\nonumber
\end{equation}
$(\romannumeral2)$ If $F=F_5(k,k,2)$, then
\begin{equation}
N_s(G)\leq  h_s(n,k,\ell).\nonumber
\end{equation}
$(\romannumeral3)$ If $F=F_0(k,k,\ell-2)$ or $F=F_4(k,k,\ell-2)$, then
\begin{equation}
N_s(G)\leq \min\{g_s(n,k,3),h_s(n,k,4),\ldots,h_s(n,k,\ell)\}.\nonumber
\end{equation}
$(\romannumeral4)$ Otherwise,
\begin{equation}
N_s(G)\leq \min\{h_s(n,k,\ell-r+1),\ldots,h_s(n,k,\ell)\}.\nonumber
\end{equation}
\end{lemma}

\begin{proof} We begin with a claim. Let
\begin{equation}
f_s(n,k,r)={k-\ell \choose s}+{\ell+1 \choose s }-{\ell-r+1 \choose s}+(n-k+\ell-r){\ell-r+1 \choose s-1}.\nonumber
\end{equation}

\medskip

\noindent{\bf Claim.}  $f_s(n,k,r)\leq h_s(n,k,t)$ for any $t\geq \ell-r+1$.

\medskip

\begin{proof}
Let $t\geq \ell-r+1$. We have
\begin{align}
f_s(n,k,r)&={k-\ell \choose s}+{\ell+1 \choose s }-{\ell-r+1 \choose s}+(n-k+\ell-r){\ell-r+1 \choose s-1}\nonumber \\
&\leq {k-t \choose s}+{t+1 \choose s}-{t \choose s}+(n-k+t-1){t \choose s-1}= h_s(n,k,t),\nonumber
\end{align}
where the second inequality follows by ${k-\ell \choose s}+{\ell+1 \choose s } \leq {k-t \choose s}+{t+1 \choose s}$ and the fact that $(n-k+t-1){t \choose s-1}-{t \choose s}$  increases with $t$ when $s\geq2$.  The proof is complete.\end{proof}

Let $F\in \mathcal{F}(m,k,r)\setminus \{F_2(k,k,r)\}$.
Let $G$ be an $n$-vertex 2-connected graph with $c(G)<k$ containing a copy of $F$ and $X=G-F$.
By Proposition~\ref{property for H}$(\romannumeral4)$, $X$ is a star forest.
First, we consider the case: $|A|=r+1$ or $k$ is odd, i.e., $F[C\cup D]$ is a $C$-path and $C,D$ are empty sets.
Since $G$ is 2-connected with $c(G)<k$, by Proposition~\ref{property for H}$(\romannumeral1)$, $(\romannumeral2)$ and $(\romannumeral3)$, it is easy to check that $X$ is an independent set.
Moreover, if $F\neq F_5(k,k,2)$, then it is easy to see that each $x\in X$ is only adjacent to $C$.
Let $t\geq \ell-r+1$. Since the numbers of unlabeled  $s$-cliques  inside $A\cup B\cup C$ and unlabeled $s$-cliques incident with $D$ are at most ${k-\ell \choose s}+{\ell+1 \choose s }-{\ell-r+1 \choose s}$ and at most $(n-k+\ell-r){\ell-r+1 \choose s-1}$ respectively.
By the claim, we have that for any $t\geq \ell-r+1$,
\begin{equation}
N_s(G)\leq {k-\ell \choose s}+{\ell+1 \choose s }-{\ell-r+1 \choose s}+(n-k+\ell-r){\ell-r+1 \choose s-1}\leq h_s(n,k,t).\nonumber
\end{equation}
Now let $F=F_5(k,k,2)$. Then each vertex of $X$ can only be adjacent to $\{u_1\}\cup C$.
Thus it is easy the check that $N_s(G)\leq  h_s(n,k,\ell)$.

Now we may suppose that $k$ is even and $|A|=r\leq \ell-3$.
Then $|C|\geq 4$.
$(a)$. $F[C\cup D]$ is a $C$-path, i.e., $F$ is of Type \2.
Then there is a unique edge in $F[D]$.
Clearly, by Proposition~\ref{property for H}$(\romannumeral1)$, $(\romannumeral2)$, and $(\romannumeral3)$, each isolated vertex of $G[X\cup D]$ is only adjacent to $C$.
Let $uw$ be the unique edge in $F[D]$.
Denoted by $u_1$  and $w_1$  the neighbours of $u$  and $w$ in $F[C\cup D]$ respectively.
Since $G$ is 2-connected with $c(G)<k$, each independent edge in $G[X\cup D]$ can only be adjacent to $\{u_1,w_1\}$.
Moreover, the center of each star $S_\alpha$ with $\alpha\geq 3$ in $G[X\cup D]$ is adjacent to both of $\{u_1,w_1\}$ and the leaves of $S_\alpha$  is only adjacent to, without loss of generality, $u_1$.
Recall that $|C|\geq 4$.
Thus, by the claim, it is not hard to show that $N_s(G)\leq f_s(n,k,\ell-r+1)\leq h_s(n,k,t)$ for any $s\geq 2$ and $t\geq \ell-r+1$.
$(b)$. $F=F_1(k,k,r)$. Since $c(G)<k$, it is easy to check that $X$ is an independent set and each $x\in X$ is only adjacent to $C$ or to $\{x_1,x_2\}$, the result follows similarly as before.
$(c)$. $F\in \mathcal{F}_3(k,k,r)$.
Let $C_1=C\cap P^\prime$  and $C_2=C\cap P^\ast$, where $P^\prime\cup P^\ast=F[C\cup D]$.
Then it is not hard to see that $X$ is an independent set.
By Proposition~\ref{property for H}$(\romannumeral1)$, $(\romannumeral2)$, and some observations (for each non-adjacent pair $(a,b)$ between $C$ and $D$, there is a path on at least $k-1$ vertices), each $x\in X$ is only adjacent to $\{z_1,z_2\}$, $\{z^\prime_1,z^\prime_2\}$, $C_1$ or $C_2$.
Hence, by the claim we have $N_s(G)\leq f_s(n,k,\ell-r+1)\leq h_s(n,k,t)$ for any $s\geq 2$ and $t\geq \ell-r+1$. The result follows.

Let $k$ be even  and $|A|=r=\ell-2$.
Then $|C|=3$.
For $F\in \mathcal{F}(m,k,\ell-2)\setminus \{F_0(k,k,\ell-2),F_2(k,k,\ell-2),F_4(k+1,k,\ell-2)\}$, similarly as previous arguments, we have $N_s(G)\leq h_s(n,k,\ell-2)$. The result follows from the claim.
Assume that $F=F_0(k,k,\ell-2)$ or $F=F_4(k+1,k,\ell-2)$.
Since $G$ is 2-connected with $c(G)<k$,  by Proposition~\ref{property for H}$(\romannumeral1)$, $(\romannumeral2)$, and $(\romannumeral3)$, each vertex of $G[D\cup X]$ is not adjacent to $A\cup B$.
Moreover, for each star $S_\alpha$ with $\alpha\geq 3$, the center of the star is adjacent to at least two vertices of $C$ and the leaves of $S_\alpha$ are adjacent to the same vertex $x\in C$.
Furthermore, for other vertices, each of them is adjacent to all vertices of $C$.
Let $t\geq 4.$
Since $n\geq k$, basic calculations show that $\left\lfloor\frac{n-k+3}{2}\right\rfloor\left({5 \choose s}-{3 \choose s}\right)+i{4 \choose s}\leq (n-k+4){4 \choose s-1}-{4 \choose s}$, where $i=1$ when $n-k+3$ is odd, and $i=0$ when $n-k+3$ is even.
Then, combining the above arguments, we have
\begin{align*}
N_s(G)\leq g_s(n,k,3)&= 2{\ell+1 \choose s}- {3 \choose s} +\left\lfloor\frac{n-k+3}{2}\right\rfloor\left({5 \choose s}-{3 \choose s}\right)+i{4 \choose s} \\
&\leq   {k-t \choose s}+{t \choose s}+(n-k+4){4 \choose s-1}-{4 \choose s}\\
&\leq   {k-t \choose s}+{t \choose s}+(n-k+t){t \choose s-1}-{t \choose s}= h_s(n,k,t),
\end{align*}
where the third inequality holds from the fact that $(n-k+t){t \choose s-1}-{t \choose s}$ increases with $t$.
Thus we finish the proof for $r=\ell-2$.

Finally, let $F=F_2(k,k,r)$.
Since $G$ is 2-connected with $c(G)<k$, by  Proposition~\ref{property for H}$(\romannumeral1)$, $(\romannumeral2)$ and $(\romannumeral3)$, $X$ is an independent set.
If  $y$ is adjacent to exactly one vertex of $A$, then each vertex in $\{y\}\cup X$ can only be adjacent to vertices of $C\cup\{y_1\}$ and each vertex of $B$ can only be adjacent to vertices of $C\cup\{y_1\}$. Similarly as the previous proof, we have $N_s(G)\leq f_s(n,k,\ell-r+2)$.
If $y$ is adjacent to two vertices of $A$, then $y$ can be adjacent to all vertices of $A$ and each vertex of $X$ can only be adjacent to $C$.
Hence, we have $N_s(G)\leq f_s(n,k,\ell-r+1)$ as before.
Thus, it follows from the claim that $N_s(G)\leq h_s(n,k,t)$ for any $t\geq \ell-r+2$.
The proof is complete.\end{proof}

\subsection{Proof of Theorem~\ref{Theorem 2-connected}}

Now we are ready for the proof of Theorem~\ref{Theorem 2-connected}.

\medskip

\noindent{\bf Proof of Theorem~\ref{Theorem 2-connected}}.
Let $k\geq 5$, $\alpha\geq 0$, $\beta\geq 2$, $\ell=\lfloor(k-1)/2\rfloor$ and $\ell-\alpha\geq \beta $.
Let $G$ be an $n$-vertex 2-connected maximal $\mathcal{K}_{k,\alpha}$-free graph with $c(G)<k$ satisfying (\ref{bound for 2-connected}).
Thus, if $xy\notin E(G)$, then either $G+xy$ contains a copy of $K\in \mathcal{K}_{k,\alpha}$, or a cycle of length at least $k$.
Now suppose that $\omega(G)\leq k-\beta $ and $|V(H(G,\ell-1))|\geq k-\ell+\alpha$. We will finish our proof by contradictions.
Let $H=H(G,\ell-1)$.

\medskip

\noindent{\bf Claim.} $H$ is a complete graph.

\begin{proof}
Suppose not, there is a non-edge $ab$ in $H$. We prove the claim in the following four cases.

\medskip

\noindent{\bf Case 1.} $G+ab$ contains a cycle of length at least $k$.

\medskip

Then, by $a,b\in V(H)$, there is an $H$-path on at least $k$ vertices.
Thus, there exists a longest $H$-path $P$ on $m\geq k$ vertices.
If $\alpha=\ell-2$ or $k\leq 8$, i.e., $\ell\leq3$, then by Lemma~\ref{extend posa lemma}, $G$ contains a copy of $F\in \mathcal{F}(m,k,r)$, contradicting that $G$ is $\mathcal{K}_{k,\alpha}$-free.
Hence, we may suppose $\alpha<\ell-2$ and $k\geq 9$, i.e., $\ell\geq 4$.
Since $G$ is $\mathcal{K}_{k,\alpha}$-free, by Lemma~\ref{extend posa lemma}, $G$ contains a copy of $F\in \mathcal{F}(m,k,r)\setminus \mathcal{K}_{k,\alpha} $.

Let $k$ be odd, or $r\leq \ell-3$ and $\alpha\geq 2$.
Since $G$ is 2-connected and $c(G)<k$, it follows from Lemma~\ref{bound the edges of F small s}$(\romannumeral4)$ that
\begin{equation}
N_s(G)\leq \min\{h_s(n,k,\ell-r+1),\ldots,h_s(n,k,\ell)\}\leq h_s(n,k,\ell-\alpha),\nonumber
\end{equation}
a contradiction to (\ref{bound for 2-connected}).
For $r\leq \ell-3$ and $\alpha=1$, combining Lemma~\ref{bound the edges of F small s}$(\romannumeral2)$ and  Lemma~\ref{bound the edges of F small s}$(\romannumeral4)$ we can also easily get a contradiction.

Assume that $r= \ell-2$ and $k\geq 10$ is even.
Note that $\ell-\alpha\geq 3$.
If $\ell-\alpha=3$, i.e., $\alpha=\ell-3$, then  we have $F\in \mathcal{F}(m,k,\ell-2)\setminus \{F_0(k,k,\ell-2),F_2(k,k,\ell-2),F_4(k+1,k,\ell-2),F_5(k,k,2)\}$ ($G$ is $\mathcal{K}_{k,\alpha}$-free).
By Lemma~\ref{bound the edges of F small s}$(\romannumeral4)$, we have
\begin{equation}
N_s(G)\leq \min\{h_s(n,k,3),h_s(n,k,4),\ldots,h_s(n,k,\ell)\}\leq h_s(n,k,3),\nonumber
\end{equation}
a contradiction.
Let $\ell-\alpha\geq 4$.
Then $F\in \mathcal{F}(m,k,\ell-2)\setminus \{F_2(k,k,\ell-2)\}$.
It follows from Lemma~\ref{bound the edges of F small s}$(\romannumeral3)$ and $(\romannumeral4)$ that
\begin{equation}
N_s(G)\leq \min\{\max\{g_s(n,k,3),h_s(n,k,3)\},h_s(n,k,4),\ldots,h_s(n,k,\ell)\}\leq h_s(n,k,\ell-\alpha),\nonumber
\end{equation}
which is also a contradiction to (\ref{bound for 2-connected}).
This completes the proof of Case 1.

\medskip

If $c(G+ab)\geq k$ or there is an $H$-path on at least $k$ vertices, then by Case 1, we get a contradiction.
Thus, in the following cases, it suffices to show that either $c(G+ab)\geq k$ or there is an $H$-path on at least $k$ vertices.

Now, suppose that $G+ab$ contains a copy of $F\in \mathcal{K}_{k,\alpha}$.
We divide the following proof into two cases basing on the value of $r$ in $\mathcal{F}(m,k,r)$.

\medskip

\noindent{\bf Case 2.} $G+ab$ contains a copy of $F\in\mathcal{F}(k,k,r)$ for some $r\in\{1,2,\ldots,\alpha\}$ with $\alpha\leq \ell-2$; $F_0(k,k,\ell-2)$ or $F_4(k+1,k,\ell-2)$ when $k\geq 10$ is even and $\ell -\alpha\leq 3$; or $F_2(k,k,\alpha+1)$ with $\alpha+1\leq \ell-2$; or $F_5(k,k,2)$ with $\alpha=1$ and $\ell\geq 4$, i.e. $r=2\leq \ell-2$.

\medskip

Let $A\cup B\cup C\cup D$ be a partition of $V(F)$ in Section 2. By Proposition~\ref{property for H}$(\romannumeral1)$, 
for each edge $ab$ of $F$, there is a path on $k-1$ vertices starting from $a$ and ending at $b$ in $F$, 
except that $ab\in \{x_1x_2,z_1z_2,z_1^\prime z_2^\prime,vv_1,vv_2,y_1y_2\}$, $ab \in E(\{u_1\},C)$ or $ab \in E(\{y_1\},C)$ with $r\geq 2$. We may assume
\begin{equation}\label{Na and Nb}
N_{H}(a)\subseteq V(F) \mbox{ and } N_{H}(b)\subseteq V(F).
\end{equation}
Otherwise,  since $a,b\in H$, by Proposition~\ref{property for H}$(\romannumeral1)$, there is an $H$-path on at least $k$ vertices, and we are done.
Note that there is no edge in $F[C]$. We can choose $v\in\{a,b\} \cap (A\cup B\cup D)$. Then we have
\begin{equation}\label{Na and Nb 2}
v\in A \mbox{ and } N_H(v)\subseteq A\cup C.
\end{equation}
Otherwise, since $|C|\leq \ell$, by Proposition~\ref{property for H}$(\romannumeral2)$ we have $c(G+ab)\geq k$, and hence $G$ contains an $H$-path on $k$ vertices.
Since $a$ is not adjacent to $b$, $N_{H}(a)\geq \ell$, $N_{H}(b)\geq \ell$ and $|A\cup C|\leq\ell+2$, it follows from (\ref{Na and Nb 2}) that each vertex in $A$ has degree at least $\ell$ in $H[A\cup C]$. Thus $G$ contains a copy of $F\in \mathcal{F}(m,k,r)$ with $|A|=r+1$ or a copy of $F(\ell)$ (when $|A|=2$ and $e(H[A])=0$). Both are contradictions.

Let $ab\in \{x_1x_2,z_1z_2,z_1^\prime z_2^\prime,vv_1,vv_2,y_1y_2\}$, $ab \in E(\{u_1\},C)$ or $ab \in E(\{y_1\},C)$ with $r\geq 2$.
Then by Proposition~\ref{property for H}$(\romannumeral1)$, there is a path on $k-2$ vertices staring from $a$ and ending at $b$ in $F$.
Thus for each $w_a\in N_H(a)\setminus V(F)$ and each $w_b\in N_H(b)\setminus V(F)$, we have
\begin{equation}\label{outside-neighbour}
w_a=w_b=w.
\end{equation}
Otherwise, there is an $H$-path starting from $w_a$ ending at $w_b$ on $k$ vertices and we are done.
Now, we consider the following six cases:

(2.1) Let $ab=x_1x_2$.
First, $a$ and $b$ are not adjacent to any vertex of $(B\cup D)\setminus \{x\}$.
Otherwise, by Proposition~\ref{property for H}$(\romannumeral2)$, we can deduce that $c(G+ab)\geq k$, and hence we are done.
Since $|A\cup C|= \ell+1$ and $a$ is not adjacent to $b$, $|N_H(a)\setminus (A\cup C)|\geq 1$ and $|N_H(b)\setminus (A\cup C)|\geq 1$.
Thus by (\ref{outside-neighbour}), we have  $N_H(a)\setminus V(F)=N_H(b)\setminus V(F)=\{w\}$.
Then $w\in V(H)$ and $A\cup C\subseteq V(H)$.
Note that each vertex of $B$ has degree $\ell$ in $G[A\cup B\cup C]$.
Thus we have $B\subseteq V(H)$.
So we can easily find an $H$-path on $k$ vertices.

(2.2) Let $ab\in\{z_1z_2,z_1^\prime z_2^\prime\}$.
Without loss of generality, let $ab=z_1z_2$.
If $N_H(z_1)\subseteq V(F)$, then Proposition~\ref{property for H}$(\romannumeral2)$ implies that $A\cup C\cup \{z\}\subseteq V(H)$.
Note that each vertex of $B$ has degree $\ell$ in $G[A\cup B\cup C]$.
This implies that $B\subseteq V(H)$.
Hence, there is a path on $k$ vertices starting from $B$ ending at $z$ and we are done.
Now we may suppose that there is a vertex $w \in N_H(z_1)\setminus V(F)$.
If there is a vertex $w^\prime \in N_H(w)\setminus V(F)$, then we can find a path on $k$ vertices starting from $w^\prime$ ending at $z_2$ and hence we are done.
Now let $N_H(w)\subseteq V(F)$.
Then by Proposition~\ref{property for H}$(\romannumeral1)$ and $(\romannumeral2)$, if there is a vertex $z\in N_H(w)\setminus\{z_1,z_2\}$, then $c(G+ab)\geq k$.
Thus   $N_H(w)\subseteq\{z_1,z_2\}$ and $\ell=2$, and hence $G=H$.
Therefore, it is easy to find an $H$-path on at least $k$ vertices.

(2.3) Let $ab\in \{vv_1,vv_2\}$.
Then $|C|=3$.
Without loss of generality, let $ab=vv_1$.
Then there is a vertex $w \in N_H(v)\setminus V(F)$.
Otherwise, since $|C|\leq\ell$ and $v$ is not adjacent to $v_1$, we have $N_H(v)\cap (A\cup B\cup D)\neq \emptyset$.
It follows from Proposition~\ref{property for H}$(\romannumeral2)$ that $c(G+ab)\geq k$ and we are done.
Again, it follows from Proposition~\ref{property for H}$(\romannumeral2)$ that $C\subseteq N_H(w)$ or there is a vertex $w^\prime\in N_H(w)\setminus V(F)$.
Thus, in the former case, we have $A\cup B\cup C\subseteq V(H)$, and hence, there is a path on $k$ vertices starting from $w$ ending at $B$.
In the later case, there is a path on $k$ vertices starting from $w^{\prime}$ ending at $v_1$.
We are done in both cases.

(2.4) Let $ab=y_1y_2$. Then $r=\alpha+1$.
$(a)$ $r\geq2$.
If $y_1$ is adjacent to $B$, then $G$ contains a copy of $F\in \mathcal{F}(k,k,\alpha)$.
If there is an $y^\prime\in N_H(y_1)\cap D$, then $y^\prime\in V(H)$; as $|C|=\ell-r+1\leq \ell-1$, there is a vertex $w^\ast\in N_H(y^\prime)$.
We can find an $H$-path on at least $k$ vertices starting from $w^\ast$ and ending at $y_1$.
Thus, there is a vertex $w\in N_H(y_1)\setminus V(F)$.
If there is a vertex $w^\prime\in N_H(w)\setminus V(F)$, then there is a path on $k$ vertices starting from $w^\prime$ ending at $y_2$, and we are done.
Assume that $N_H(w)\subseteq V(F)$.
Thus by Proposition~\ref{property for H}$(\romannumeral1)$, we have $N_H(w)\subseteq C$.
Therefore, since $|C|\leq \ell-1$, we have $C= N_H(w) \subseteq V(H)$.
Hence, as before, we have $B\subseteq V(H)$.
We can easily find a path on $k$ vertices starting from $w$ and ending at $B$.
$(b)$ $r=1$. This case is similar as $(a)$.

For $ab \in E(\{y_1\},C)$ with $r\geq 2$ or $ab \in E(\{u_1\},C)$, 
the proofs are essentially the same as the proof of (2.4) and thus we omit here.

\medskip

\noindent{\bf Case 3.} $G+ab$ contains a copy of $F\in \mathcal{F}(m,k,r)$ for $r\in\{\ell-1,\ell\}$.

\medskip

Let $G+ab$ contains a copy of $F\in \mathcal{F}(m,k,\ell-1)$.
Note that $\delta(F[A\cup  B\cup C])\geq \ell$ and $a,b\in  V(H)$.
It is clearly that $A\cup B \cup C\subseteq V(H)$.
For $k$ is odd, or $|A|=r=\ell-1$ and $k$ is even, since $F[A]=F[B]=K_{\ell-1}$, it is not hard to find an $H$-path on at least $k$ vertices.
Let $k$ be even  and $|A|=r+1=\ell$.
$(a)$. $ab$ is incident with $D$. Let $a\in D$.
Since $d_{F[A\cup C]}(w)\geq \ell$ for each $w\in A$, by Proposition~\ref{basic property for A=r+1}$(\romannumeral1)$ when $\ell\geq 4$ and by definition of $\mathcal{F}(m,k,\ell-1)$ when $\ell=2$, we can find a Hamilton $C$-path in $F[A\cup C]$.
For $\ell=3$, a simple observation shows that there is also a Hamilton $C$-path in $F[A\cup C]$.
Thus there is path on at least $k$ vertices starting from $a$, through the Hamilton path in $F[A\cup C]$, ending at $B$.
We are done.
$(b)$. $ab \in F[A\cup C]$ or $ab \in F[B\cup C]$.
Note that $A\cup B \cup C\subseteq V(H)$. 
The proofs (to be divided into cases: $ab\in \{x_1x_2,y_1y_2\}$, $ab \in E(\{u_1\},C)$ or $ab \in E(\{y_1\},C)$ with $r\geq 2$.) can be handled similarly as the proofs in Case 2.

Now, let $k$ be even and $G+ab$ contains a copy of $F\in \mathcal{F}(m,k,\ell)$ for $m\geq k$.
Let $X=A\cup B\cup \{w,w_1,w_2\}$.
Since the degree of each vertex of $X\setminus\{a,b\}$ in $G[X]$ is at least $\ell$, together with $a,b\in H$, we have $X\subseteq H$.
Hence, if $ab\notin E(C)$ or $ab=w_1w_2$, then we can easily find an $H$-path on at least $k$ vertices.
If $ab\in E(C)$ and $ab\neq w_1w_2$, then there is a path on at least $k$ vertices starting from $a$ (or $b$), through $w$, $A$, $w_1w_2$ and ending at $B$ (note that $ww_1,ww_2 \notin E(C)$).

\medskip

\noindent{\bf Case 4.} $G+ab$ contains a copy of $F(\ell)$ when $k$ is even.

\medskip

Let $A$ and $B$ be the partite sets of $P_{2\ell-1}$ in $F(\ell)$ with $|A|=\ell$  and $|B|=\ell-1$.
Let $C=V(F(\ell))\setminus(A\cup B)$.
If there is an edge in $G[C]$, then $G+ab$ contains a copy of $F\in \mathcal{F}(k,k,1)$, and we are done by Case 2.
Thus, we may assume that $G[C]$ is empty.
Let $a\in A$ and $b\in C$.
If $b$ is adjacent to $B$, then $G+ab$ contains a copy of $F\in \mathcal{F}(k,k,1)$, and hence we are done.
Since $|C|\leq \ell$ and $a$ is not adjacent to $b$, there is a vertex $w\in N_H(b)\setminus (A\cup B\cup C)$.
If there is a $w^\prime \in N_H(w)\setminus (A\cup B\cup C)$, then there is a path on $k$ vertices starting from $w^\prime$ ending at $a$.
Thus, $N_H(w)\subseteq(A\cup B\cup C)$.
Then, it is not hard to see that $c(G+ab)\geq k$, and we are done.
We omit the proof when $a\in A$ and $b\in B$.
We complete our proof of the claim.
\end{proof}

Let $|V(H)|=m.$ Since $V(H(G,\ell-1))\geq k-\ell+\alpha$ and $\omega(G)\leq k-\beta$, we have $k-\ell+\alpha\leq m\leq k-\beta$.
Apply to the graph $G$ the process of $(k-m)$-disintegration. Let $H^\prime=H(G,k-m)$. If $H^\prime=H$, then
\begin{align*}
N_s(G)&\leq {m \choose s}+ (n-m){k -m-1 \choose s-1}\leq \max\{h_s(n,k,\ell-\alpha),h_s(n,k,\beta)\},
\end{align*}
a contradiction to (\ref{bound for 2-connected}). If $H^\prime\neq H$, then there exists a vertex $b\in V(H^\prime)$ which is not adjacent to a vertex $a\in V(H)$.
We divide the proof into the following two cases:
$(a)$. Adding $ab$, the obtained graph contains a cycle of length at least $k$. Then there is a path in $G$ on at least $k$ vertices starting in $H$ and ending in $H^\prime$.
Let $P=xPy$ be a longest such path with $x\in V(H)$ and $y\in V(H^\prime)$.
Then we have $d_P(a)\geq m-1$ and $d_P(b)\geq k-m+1$. It follows from  Lemma~\ref{posa lemma} that $c(G)\geq k$, a contradiction.
$(b)$. Adding $ab$, the obtained graph contains a copy of $K\in \mathcal{K}_{k,\alpha}$.
Note that $H$ is a complete graph on $m\geq k-\ell+\alpha$ vertices, $d_{H^\prime}(b)\geq k-m+1$ and $c(G)\leq k-1$.
Similarly as the Cases 2, 3 and 4, we can find a path on at least $k$ vertices starting from $H$ and ending at $H^\prime$ (actually the situation here will be easier than previous cases).
Thus, by Lemma~\ref{posa lemma} again, we have $c(G)\geq k$.
This final contradiction completes the proof of Theorem~\ref{Theorem 2-connected}. \hfill $\square$

\section{Implications}\label{corollary}

In this section, we shall use Theorem~\ref{Theorem 2-connected} to deduce Theorem~\ref{corollary 2-connected large n} and equivalent statements of some main results in \cite{Furedi2016,Furedi2018,MN2019}.
We need the following result proved by Fan~\cite{Fan1990}.

\begin{theorem}[Fan \cite{Fan1990}]\label{Fan theorem}
Let $G$ be an $n$-vertex 2-connected graph and $ab$ be an edge in $G$.
If the longest path starting from $a$ and ending at $b$ in $G$ has at most $r$ vertices, then $e(G)\leq \frac{(r-3)(n-2)}{2}+2n-3.$
Moreover, the equality holds if and only if $G-\{a,b\}$ is a vertex-disjoint union of copies of $K_r$.
\end{theorem}

First, we can derive a more general result concerning the number of cliques from Theorem~\ref{Theorem 2-connected}.

\begin{corollary}\label{corollary 2-connected}
Let $G$ be an $n$-vertex 2-connected graph with minimum degree $\delta\geq 2.$
Let $k\geq 9$ and $\ell-1\geq \delta+1$.\footnote{If $5\leq k\leq 8$, then $\ell-1<\delta+1$. By (\ref{eq for corollary}), it follows from Luo's theorem that $G$ contains a cycle of length at least $k$.}
If $c(G)<k$ and
\begin{equation}\label{eq for corollary}
N_s(G)>\max\{h_s(n,k,\ell-1),h_s(n,k,\delta+1)\},
\end{equation}
then one of the following holds:\\
 $(a)$ $G$ contains a copy of   \\
\indent $(a.1)$ $F\in \mathcal{F}(m,k,r)$ with  $r\in\{1,\ell-1,\ell\}$, or\\
\indent $(a.2)$  $F_0(10,10,2)$ or $F_4(11,10,2)$ when $k=10$, or\\
\indent $(a.3)$  $F_2(m,k,2)$ or $F_5(m,k,2)$ when $k$ is even, or\\
\indent $(a.4)$  $F(\ell)$ when $k$ is even;\\
 $(b)$ $G$ is a subgraph of the graph $Z(n,k,\delta)$;\footnote{The graph $Z(n,k,\delta)$ denotes the vertex-disjoint union of a clique $K_{k-\delta}$ and some cliques $K_{\delta+1}$'s, where any two cliques share the same two vertices.}\\
 $(c)$ $G$ is a subgraph of $H(n,k,\delta)$.
\end{corollary}
\begin{proof}
If (a) holds, then we are done.
Thus we may suppose that $G$ is $\mathcal{K}_{k,1}$-free.
Let $J\supseteq G$ be a maximal $\mathcal{K}_{k,1}$-free with $c(J)<k$.
Suppose that $|V(H(J,\ell-1)|\leq k-\ell$. Then \begin{align*}
N_s(J)\leq (n-k+\ell){\ell-1\choose s-1}+{k-\ell \choose s}      =h_s(n,k,\ell-1)
\end{align*} contradicting (\ref{eq for corollary}).
Thus we have $|V(H(J,\ell-1)|\geq k-\ell+1$.
Clearly, $N_s(J)>\max\{h_s(n,k,\ell-1),h_s(n,k,\delta+1)\}$ and $\delta(J)\geq \delta.$
Applying Theorem~\ref{Theorem 2-connected} with $\alpha=1$ and $\beta=\delta +1$, we have that $\omega(J)\geq k-\delta$.

It suffices to show that either $(b)$ or $(c)$ holds.
Let $K$ be a maximum clique in $J$.
Then there is a non-edge $x_1x_m$ with $x_1\in K$ and $x_m\in J-K$.
Then by the maximality of $J$, we may first suppose that there is a longest path $P=x_1x_2\ldots x_m$ on $m\geq k$ vertices starting from $x_1$ and ending at $x_m$.
Thus we have $d_P(x_1)\geq k-\delta-1$ (note that $x_1\in K$  and $|K|\geq k-\delta$) and $d_P(x_m)\geq \delta$.
Similarly as the proofs in Lemma~\ref{extend posa lemma}, we only need consider the case that there exist $i$ and $j$ with $2\leq i< j \leq m-1$ such that $x_j \in N_P(x_1)$ and $x_i\in N_P(x_m)$.
Moreover, by  the proofs of Lemma~\ref{extend posa lemma} (since there is a clique of size $k-\delta$ in $G[V(P)]$, the proofs here are much easier), $J$ contains a copy of $F\in \mathcal{F}(m,k,r,\delta)$ for $1\leq r\leq \delta-1$, where each  $F\in \mathcal{F}(m,k,r,\delta)$ with a partition $V(F)=A\cup B\cup C\cup D$ on $m$ vertices satisfies the following:
\begin{itemize}
  \item $F[A]$ and $F[B]$ are  complete graphs with $|A|=r$ and $|B|=k-2\delta +r-1$;
  \item $F[C]$ is empty with $|C|=\delta-r+1$;
  \item $F(A,C)$ and $F(B, C)$  are  complete bipartite graphs;
  \item $F[D]$ is empty or a path when  $m\geq k+1$ and $|C|=2$;
  \item and $F[C\cup D]$ is a $C$-path.
\end{itemize}
The graph family $\mathcal{F}(m,k,r,\delta)$ plays the same role as the graph family $\mathcal{F}(m,k,r)$, that is, if a 2-connected graph with $c(J)<k$ containing a copy of $F\in\mathcal{F}(m,k,r,\delta)$, then each component of $J-V(F)$ can only be adjacent to $C$ of $V(F)$.
Thus, if $|C|\geq 3$, then $J-V(F)$ is an independent set.
Since $\delta(J)\geq\delta$, we have either $m=k$, $r=1$ and $|C|=\delta$, or $r=\delta-1$ and $|C|=2$.
If $m=k$, $r=1$ and $|C|=\delta$, then $J$ and hence $G$ are  subgraphs of $H(n,k,\delta)$.
Let $r=\delta-1$ and $|C|=2$.
Note that $J$ is 2-connected with $c(J)<k$.
Each path starting from $C$ ending at $C$ is on at most $\delta+1$ vertices.
Since $\delta(J)\geq\delta$, it follows from a result of Erd\H{o}s and Gallai \cite{erdHos1959maximal} (a minimum degree version of Theorem~\ref{Fan theorem}) that $J-A\cup B\cup C$ is a union of copies of $K_{\delta-1}$.
Thus $J$ and hence $G$ are subgraphs of $Z(n,k,\delta)$, i.e., the union of a clique $K_{k-\delta}$ and some cliques $K_{\delta+1}$'s, where any two cliques share the same two vertices.

Now we suppose that $J+x_1x_m$ contains a copy of $\mathcal{K}_{k,1}$, then it is easy (as in the proof of Theorem~\ref{Theorem 2-connected}) to find a path on at least $k$ vertices starting from $K$ and ending at $J-K$.
The result follows by applying the previous argument. \end{proof}

Now we are ready to deduce Theorem~\ref{corollary 2-connected large n} from Theorem~\ref{Theorem 2-connected}.

\medskip

\noindent{\bf Proof of Theorem~\ref{corollary 2-connected large n}.}
Let $2\leq s\leq \max\{2, \ell-1\}$ and $G$ be an $n$-vertex 2-connected {\it maximal} (in the sense that adding any edge will create a cycle of length at least $k$) graph with $c(G)<k$ and $N_s(G)> h_s(n,k,\ell-1)$.
Note that if $G$ satisfies the conclusion of Theorem~\ref{corollary 2-connected large n}, then any subgraph of $G$ also satisfies the conclusion of Theorem~\ref{corollary 2-connected large n}.
Thus, it is suffices to prove Theorem~\ref{corollary 2-connected large n} for the maximal graph $G$.
Similarly as  the proof of Corollary~\ref{corollary 2-connected}, we have $|V(H(G,\ell-1)|\geq k-\ell+1$.
Since $G$ is 2-connected with $c(G)<k$, we have $\omega(G)\leq  k-2$.
For $\ell-1\geq 2$, i.e., $k\geq7$ and $\ell\geq 3$, basic calculations show  that there is a constant $n_0$ such that if $n>n_0$, then $h_s(n,k,\ell-1)\geq h_s(n,k,2)$.
Let $n\geq n_0$ be sufficiently large and $m\geq k$.
Combining the above arguments and applying Theorem~\ref{Theorem 2-connected} with $\alpha=1$ and $\beta=2$ when $\ell\geq 3$, and with $\alpha=1$ and $\beta=1$ when $\ell=2$, $G$ contains a copy of $F\in \mathcal{F}(m,k,r)$ with $r\in \{1,\ell-1,\ell\}$ or $F\in \{F_2(k,k,2),F_5(k,k,2),F_0(10,10,2),F_4(11,10,2),F(\ell)\}$.
If $F\notin\mathcal{F}(m,k,r)$ with $r=\ell-1,\ell$, then similarly as the proof of Proposition~\ref{property for H}$(\romannumeral4)$, it is easy to check that there is an $X\subseteq V(G)$ of order at most $\ell$ such that $G-X$ is a star forest.
Moreover, if $k$ is odd, then $G-X$ is an independent set and each vertex of it can only be adjacent to $X$, where $X \subseteq V(F)$ is of size $\ell$.

Assume that $G$ contains a copy of $F\in \mathcal{F}(m,k,r)$ for $r\in\{\ell-1,\ell\}$.
For $\ell\leq 3$, again, the result follows similarly as the proof of Proposition~\ref{property for H}$(\romannumeral4)$.
Furthermore, if $k=5$, then there is an $X\subseteq V(G)$ of order two such that $G-X$ is an independent set;
if $k=7$, then there is an $X\subseteq V(G)$ of order two such that $G-X$ is a star forest.
Let $\ell\geq 5$ when $s=3$ and $\ell\geq 4$ otherwise. We need the following fact that

$$
(\ell-1){\ell-1 \choose s-1}\geq {\ell+1 \choose s}-{2 \choose s} \mbox{ holds  for } 2\leq s\leq \ell-1.
$$

To see this, first let $s=2$. Since $\ell\geq 4$, we have $(\ell-1)^2\geq {\ell+1 \choose s}-1.$
Now let $3\leq s\leq \ell-1.$ Then it is enough to show that $s(\ell-1)(\ell-s+1)\geq (\ell+1)\ell$. It is easy to see that $s(\ell-1)(\ell-s+1)$ attains its minimum when $s=3$ or $s=\ell-1$.
If $s=3$, then by $\ell\geq 5$, we have $3(\ell-1)(\ell-2)\geq (\ell+1)\ell$. If $s=\ell-1>3$, then $\ell\geq 5$,
and hence we have $2(\ell-1)^2> (\ell+1)\ell$, proving this fact.

Let $n\geq n_0+t\ell-t$, where $t$ is a large constant. Since $G$  contain a copy of $F\in \mathcal{F}(m,k,r)$ for $r\in\{\ell-1,\ell\}$.
In both cases, $F$ contains a clique with size $\ell-1$ such that after deleting the vertices of it, the resulting graph $G_1$ is 2-connected.
Moreover, each vertex of the clique is only adjacent to two vertices of $F$ ($C$ when $F\in \mathcal{F}(k,k,\ell-1)$, and $w,w_1$ or $w,w_2$ when $F\in \mathcal{F}(k,k,\ell)$).
Thus, by $s\leq\ell-1$ and the above fact we have
$$N_s(G_1)\geq h_s(n,k,\ell-1)- \left({\ell+1 \choose s}-{2 \choose s}\right)>h_s(n-\ell+1,k,\ell-1).$$
Repeat this progress $t-1$ times. Since $t$ is sufficiently large, we have
$$N_s(G_{t})\geq h_s(n,k,\ell-1)-t\left({\ell+1 \choose s}-{2 \choose s}\right)> N_s(K_{n_0+\ell-1}),$$
a contradiction.  The proof of Theorem~\ref{corollary 2-connected large n} now is complete. \hfill $\square$\par\medskip

We remark that for the case $s=3$ and $k\in \{9,10\}$ in the conclusion of Theorem~\ref{corollary 2-connected large n},
one can obtain a refined structural description by deleting at most four vertices so that the resulting graph is very close to a star forest
(there may be some triangles in the resulting graph). 

Next, we show how to use the above results to deduce some of the main results in \cite{Furedi2016,Furedi2018,MN2019} (in equivalent forms).
We need the following lemma.

\begin{lemma}\label{bound the edges of F large s}
For $n\geq k\geq 9$, let $G$ be an $n$-vertex 2-connected graph with $c(G)<k$.
If $G$ contains a copy of $F\in\mathcal{F}(m,k,r)$ with $r\in\{\ell-1,\ell\}$, then
$e(G)\leq h_2(n,k,\ell-1)$.
\end{lemma}
\begin{proof} Let $F\in\mathcal{F}(m,k,\ell-1)$ and  $C=\{c_1,c_2\}$.
Let $k=2\ell+1$ be odd.
Then it is easy to see that the longest path starting from $c_1$ ending at $c_2$ is on at most $\ell+1$ vertices.
Since $\ell\geq 4$ and $n\geq k$, by Theorem~\ref{Fan theorem},

\begin{align*}
e(G)\leq \frac{(\ell-2)(n-2)}{2}+2n-3<{\ell+2 \choose 2}+(\ell-1)(n-\ell-2)=h_2(n,k,\ell-1),
\end{align*}
as desired.
Let $k=2\ell+2$ be even and $\ell\geq 5$.
Note that the longest path starting from $c_1$ ending at $c_2$ in $G$ is on $\ell+2$ vertices (if there is a path starting from $c_1$ ending at $c_2$ in $G$ on $\ell+3$ vertices, then one may easily check that $c(G)\geq k$ by $G$ contains a copy of $F$, a contradiction).
Then by Theorem~\ref{Fan theorem} and $\ell\geq 5$, we have
\begin{align*}
e(G)\leq \frac{(\ell-1)(n-2)}{2}+2n-3<{\ell+3 \choose 2}+(\ell-1)(n-\ell-3)=h_2(n,k,\ell-1).
\end{align*}
Thus we may suppose $k=10$. If the longest path starting from $c_1$ and ending in $c_2$ has at most five vertices, then  by Theorem~\ref{Fan theorem}, we have
$$e(G)\leq \frac{2(n-2)}{2}+2n-3< 3n-3=h_2(n,10,3).$$
Thus we may assume that there is a longest path $P=c_1x_1x_2x_3x_4c_2$ in $G$. Let $G^\prime= G-\{c_1,c_2\}$. Let $X$ be the component of $G^\prime$ contains $\{x_1,x_2,x_3,x_4\}$. Let $X^\prime=X\setminus \{x_1,x_2,x_3,x_4\}$.
Let $C$ be a component of $G[X^\prime]$ and $P^\ast=y_1y_2\ldots y_s$ be a longest path in $C$ such that $y_1$ and $y_s$ are adjacent to distinct vertices of $P$ respectively.
First we show that $s\leq 3$.
Assume that $s\geq 4$.
Since the longest path starting from $c_1$ and ending in $c_2$ has at most six  vertices, $y_1$ and  $y_s$ are adjacent to $c_1$ and $c_2$, respectively.
Note that $P^\ast$ is connected to $\{x_1,x_2,x_3,x_4\}$.
We can easily find a path starting from $c_1$ and ending in $c_2$ has at least seven vertices, a contradiction.
If $s=1$, then $C$ is an isolated vertex.
Hence, the number of edges incident with $C$ in $G$ is at most three.
If $s=2$, then $y_1$ and $y_2$ is adjacent to two vertices of $P^\ast$ with distant at least three  respectively.
As the proof Proposition~\ref{property for H}$(\romannumeral4)$, we can show that $C$ is a star and hence the number of edges incident with $C$ in $G$ is at most $2|C|$.
At last, let $s=3$.
Tedious analysis shows that $C$ is $K_{2,|C|-2}$, a star or a triangle.
In all of the above cases, the number of edges incident with $C$ in $G$ is at most $3|C|$.

For any other component $Y$ of $G^\prime$, since $c(G)<10$, by Theorem~\ref{Fan theorem}, the number edges incident with it is at most $3|Y|$.
Summing all the above edges, we have $e(G)\leq {6 \choose 2}+3(n-6)=3n-3= h_2(n,10,3)$. We finish the proof when $F\in\mathcal{F}(m,k,\ell-1)$.

Let $F\in\mathcal{F}(m,k,\ell)$. Then $k$ is even and $n\geq 10$.
Let $w,w_1$ and $w_2$ be the vertices of $F$ as in Section~\ref{notation}.
Since $c(G)<k$, the longest path starting from $w_1$ or $w_2$ through $G-\{w,w_1,w_2\}$ ending at $w$ is on at most $\ell+1$ vertices and each component of $G-\{w,w_1,w_2\}$ can only be adjacent to $w_1,w$ or $w_2,w$. Let $G_i$ be the induced subgraph of $G$ containing $\{w,w_i\}$ and all components of $G-\{w,w_1,w_2\}$ which is adjacent to $w_i$ for  $i=1,2$. Let $n_1=|V(G_1)|$ and $n_2=|V(G_2)|$. Then $n=n_1+n_2-1$. Since $n\geq 10$ and $\ell\geq 4$, by Theorem~\ref{Fan theorem}, we have
\begin{align*}
e(G)&= e(G_1)+e(G_2)+1\leq \frac{(\ell-2)(n_1-2)}{2}+2n_1-3+\frac{(\ell-2)(n_2-2)}{2}+2n_2-3+1\\
&= \frac{(\ell-2)(n-3)}{2}+2(n+1)-3-2<{\ell+3 \choose 2}+(\ell-1)(n-\ell-3)=h_2(n,k,\ell-1).
\end{align*}
This finishes the proof of the lemma.\end{proof}

Now we have the following immediate corollary, which can imply some of the main results in \cite{Furedi2016,Furedi2018,MN2019}.
\begin{corollary}\label{corollary 2-connected 1}
Let $G$ be an $n$-vertex 2-connected graph with $c(G)<k$ and minimum degree $\delta(G)=\delta$. Let $k\geq 9$ and $\ell-1\geq \delta+1$. If
\begin{equation*}
e(G)>\max\{h_2(n,k,\ell-1),h_2(n,k,\delta+1)\},
\end{equation*}
then one of the following holds:\\
$(a)$ $G$ contains a copy of $F\in \mathcal{F}(k,k,1)\cup \{F_2(k,k,2),F_5(k,k,2),F_0(10,10,2),F_4(11,10,2),F(\ell)\}$;\\
$(b)$ $G$ is a subgraph of $Z(n,k,\delta)$;\\
$(c)$ $G$ is a subgraph of $H(n,k,\delta)$.
\end{corollary}
\begin{proof}  This result follows directly from Corollary~\ref{corollary 2-connected} with $s=2$ and Lemma~\ref{bound the edges of F large s}.\end{proof}

We would like to briefly explain how this corollary can imply the main result in \cite{MN2019} (i.e., Theorem~1.7).
In our setting, Theorem~1.7 of \cite{MN2019} states that under the same conditions, if in addition $k\geq 11$,
then one of the following four cases holds: $(1)$ $G$ is a subgraph of $H(n,k,\delta)$;
$(2)$ $G$ is a subgraph of $H(n,k,\ell)$;
$(3)$ if $\delta=2$ and $k$ is even, then $G$ is a subgraph of a member of two well-characterized families of graphs;
$(4)$ if $\delta\geq 3$, then $G$ is a subgraph of $Z(n,k,\delta)$.
It is clear that $(c)$ and $(b)$ in Corollary~\ref{corollary 2-connected} correspond to the above $(1)$ and $(4)$, respectively.
Without referring to uncomplicated details, we point out that if $G$ contains a copy of $F\in \mathcal{F}(k,k,1)\cup \{F_2(k,k,2),F_5(k,k,2),F(\ell)\}$, then this would lead to the above $(2)$ and $(3)$.\footnote{Here, as $k\geq 11,$ the copy $F$ cannot be $F_0(10,10,2)$ or $F_4(11,10,2)$.}


We also can extend our results to connected graphs without paths of a given length.
Let $G$ be an $n$-vertex connected graph without containing a path of length $k-2$.
Let $G^\ast$ be the graph obtained from $G$ by adding a new vertex $v$ and joining $v$ to all vertices of $G$.
Then $G^\ast$ is an $(n+1)$-vertex 2-connected graph containing no cycle of length at least $k$.
Now using a similar argument as in \cite{Luo2018} (i.e., consider the unlabeled $s$-cliques without containing $v$),
one can prove the following result as an analogous path version of Theorem~\ref{corollary 2-connected large n}.

\begin{corollary}
Let $k\geq 5$, $2\leq s\leq \max\{2, \ell-2\}$ and $n\geq n_0(\ell)$, where $n_0(\ell)$ is a large constant depending on $\ell$.
Let $G$ be an $n$-vertex connected graph without containing path of length $k-2$.
Then $N_s(G)\leq h_s(n,k-2,\ell-2)$ unless
\begin{itemize}
\item [(a)] $s=3$ and $k\in \{9,10\}$,
\item [(b)] $k=2\ell+1$, $k\neq 7$, and $G\subseteq H(n,k-2,\ell-1)$, or
\item [(c)] $k=2\ell+2$ or $k=7$, and $G-A$ is a star forest for some $A\subseteq V(G)$ of size at most $\ell-1$.
\end{itemize}
\end{corollary}

To conclude this paper, we would like to propose the following conjecture.
This (if true) would give a strengthening of Theorem \ref{corollary 2-connected large n}
(to cover all ranges of $n$ similarly as in Corollary~\ref{corollary 2-connected 1}).

\begin{conjecture}
Let $G$ be a $2$-connected graph on $n$ vertices and let $ab$ be an edge in $G$.
Let $r\geq 4$ and $s\geq 2$ be integers, and let $n-2=x(r-3)+t$ for some $0\leq t\leq r-4$.
If $N_s(G)>x {r-1 \choose s}+ {t+2\choose s}$, then there is a cycle on at least $r$ vertices containing the edge $ab$.
\end{conjecture}

This also can be viewed as a clique version of Theorem~\ref{Fan theorem} of Fan \cite{Fan1990}.

\medskip

\bigskip

\noindent {\bf Acknowledgement.}
The authors would like to thank Alexandr Kostochka and Ruth Luo for helpful discussions at early stage of this study and Qingyi Huo for his careful reading on a draft.


\begin{thebibliography}{1}
\bibitem{bollobas1978} B. Bollob\'{a}s, Extremal graph theory, Academic press 1978.
\bibitem{erdHos1959maximal} P.~Erd\H{o}s and T.~Gallai, On maximal paths and circuits of graphs, {\em Acta Mathematica Hungarica} {\bf 10(3)} (1959), 337--356.
\bibitem{Fan1990} G.~Fan, {Long cycles and the codiameter of a graph, I}, {\em J. Combin. Theory Ser. B} {\bf 49} (1990), 151--180.
\bibitem{Furedi2016} Z.~F\"{u}redi, A.~Kostochka and J.~Verstra\"{e}te, Stability in the Erd\H{o}s-Gallai Theorem on cycles and paths, {\em J. Combin. Theory Ser. B} {\bf 121} (2016), 197--228.
\bibitem{Furedi2018} Z.~F\"{u}redi, A.~Kostochka, R.~Luo and J.~Verstra\"{e}te, Stability in the Erd\H{o}s-Gallai Theorem on cycles and paths, II, {\em Discrete Math.} {\bf 341} (2018), 1253--1263.
\bibitem{Kopylov1977} G.~N.~Kopylov, On maximal paths and cycles in a graph, {\em Soviet Math. Dokl.} {\bf 18} (1977), 593--596.
\bibitem{Luo2018} R.~Luo, The maximum number of cliques in graphs without long cycles, {\em J. Combin. Theory Ser. B} {\bf 128} (2018), 219--226.
\bibitem{MN2019} J.~Ma and B.~Ning, Stability results on the circumference of a graph, {\em Combinatorica} {\bf 40} (2020), 105--147.
\bibitem{Posa1962} L.~P\'{o}sa, A theorem concerning Hamilton lines, {\em Magyar Tud. Akad. Mat. Kutat\'{o} Int. K\"{o}zl.} {\bf 7} (1962) 225--226.
\end{thebibliography}
\end{document}